\documentclass[11pt,a4paper]{amsart}
\usepackage{amsmath,amssymb,amscd,url,epsf,bbm,mathrsfs,a4wide}
\usepackage{graphicx,url,epsf,accents}
\newtheorem{theorem}{Theorem}
\newtheorem{prop}{Proposition}
\newtheorem{corollary}{Corollary}
\newtheorem{lemma}{Lemma}
\newcommand{\T}{{\rm pc}}

\newcommand{\nts}{\hspace{-0.5pt}}
\newcommand{\ts}{\hspace{0.5pt}}
\newcommand{\udo}[1]{\underaccent{$\text{.}$}{#1\ts}\nts}

\title[Ergodic properties of visible lattice points]
{Ergodic properties of visible lattice points}

\author{Michael Baake}
\author{Christian Huck}
\address{Fakult\"at f\"ur Mathematik, Universit\"at Bielefeld,
  Postfach 100131, Bielefeld, Germany}
\email{\{mbaake,huck\}@math.uni-bielefeld.de}

\begin{document}


\setlength{\unitlength}{1mm}

\begin{abstract}
Recently, the dynamical and spectral properties of square-free
integers, visible lattice points and various generalisations have
received increased attention. One reason is the connection of
one-dimensional examples such as $\mathscr B$-free numbers with
Sarnak's conjecture on the `randomness' of the M\"obius function,
another the explicit computability of correlation functions as well as
eigenfunctions for these systems together with intrinsic ergodicity
properties. Here, we summarise some of the results, with focus on
spectral and dynamical aspects, and expand a little on the
implications for mathematical diffraction theory.
\end{abstract}

\maketitle

\centerline{Dedicated to Nikolai P.~Dolbilin on the occasion of his 70th birthday}


\section{Introduction}

Delone sets are important mathematical descriptions of atomic
arrangements, and Meyer sets are special cases with a rich spectral
structure~\cite{TAO,DD,Dolbilin}. Particularly well-studied are cut and
project sets or \emph{model}\/ sets, which underly the structure of
perfect quasicrystals~\cite{Shechtman,Steurer,TAO}. It is fair to say
that such structures are rather well understood. This is much less so
if one keeps uniform discreteness but relaxes relative denseness. In
fact, one might expect to leave the realm of pure point spectrum, at
least as soon as one has entropy~\cite{BLR}. However, there are
interesting examples such as $k$-free numbers or visible lattice
points that have positive topological entropy but nevertheless pure
point dynamical and diffraction spectrum. They are examples of
\emph{weak} model sets, and deserve a better understanding. Here, we
summarise some of the known results and add to the structure of their
topological and spectral properties. 

The paper is organised as follows. In Section~\ref{visible}, we use the visible points of $\mathbb Z^2$ as a
paradigm to formulate the results for this case in a geometrically
oriented 
manner and to develop our notation and methods while we go along. Section~\ref{kfree} extends the findings to $k$-free points of
$n$-dimensional lattices, while Section~\ref{bfree} looks into the
setting of $\mathscr B$-free systems as introduced
in~\cite{ELD,KLW}. Finally, in Section~\ref{number}, we analyse an
example from the number field generalisation of~\cite{CV} in our more
geometric setting of diffraction analysis. For convenience and better readability, the more technical issues are
presented in
two appendices.

\section{Visible square lattice points}\label{visible}

Two classic examples for the structure we are after are provided
by the \emph{square-free integers}\/ (the elements of $\mathbb Z$ that are
not divisible by any nontrivial square) and the \emph{visible
  points}\/ of a lattice (the points with coprime coordinates in a
lattice basis). Since our focus is on higher-dimensional cases, let us
take a closer look at the visible points of $\mathbb Z^2$,
$$
V\,=\,V_{\mathbb Z^2}\,=\,\mathbb
Z^2\setminus\bigcup_{\text{\scriptsize $p$ prime}} p\mathbb
Z^2\,=\,\{x\in\mathbb Z^2|\operatorname{gcd}(x)=1\},
$$
where $\operatorname{gcd}(x)=\operatorname{gcd}(x_1,x_2)$ for
$x=(x_1,x_2)$. Note that, throughout the text, $p$ will denote a prime
number. The set is illustrated in Fig.~\ref{fig: visible}, and
can also be found in many textbooks including~\cite{Apostol}, where it
is shown on the cover, and~\cite[Sec.\ 10.4]{TAO}. The following
result is standard; see~\cite[Prop.\ 10.4]{TAO} and references
therein for details.

\begin{prop}\label{propbasic}
The set\/ $V$\/ has the following properties.
\begin{itemize}
\item[(a)]
The set\/ $V$ is uniformly discrete, but not relatively
dense. In particular,\/ $V$ contains holes of arbitrary size
that repeat lattice-periodically. More precisely, given an inradius
$\rho>0$, there is a  sublattice of\/
$\mathbb Z^2$ depending on $\rho$ such that a suitable translate of this
sublattice consists of centres of holes of inradius at least $\rho$.
\item[(b)]
The group\/ $\operatorname{GL}(2,\mathbb Z)$ acts transitively on
$V$, and one has the partition\/ $\mathbb
Z^2=\dot\bigcup_{m\in\mathbb N_0}mV$ of\/ $\mathbb Z^2$ into
$\operatorname{GL}(2,\mathbb Z)$-invariant sets.
\item[(c)]
The difference set is\/\/ $V-V=\mathbb Z^2$.
\item[(d)]
The natural density of\/ $V$ exists and is given by
$\operatorname{dens}(V)=\frac{1}{\zeta(2)}=\frac{6}{\pi^2}$,
where\/ $\zeta$ denotes Riemann's zeta function.\hfill\qed
\end{itemize}
\end{prop}

\begin{center}
\begin{figure}
\includegraphics[width=0.84\textwidth]{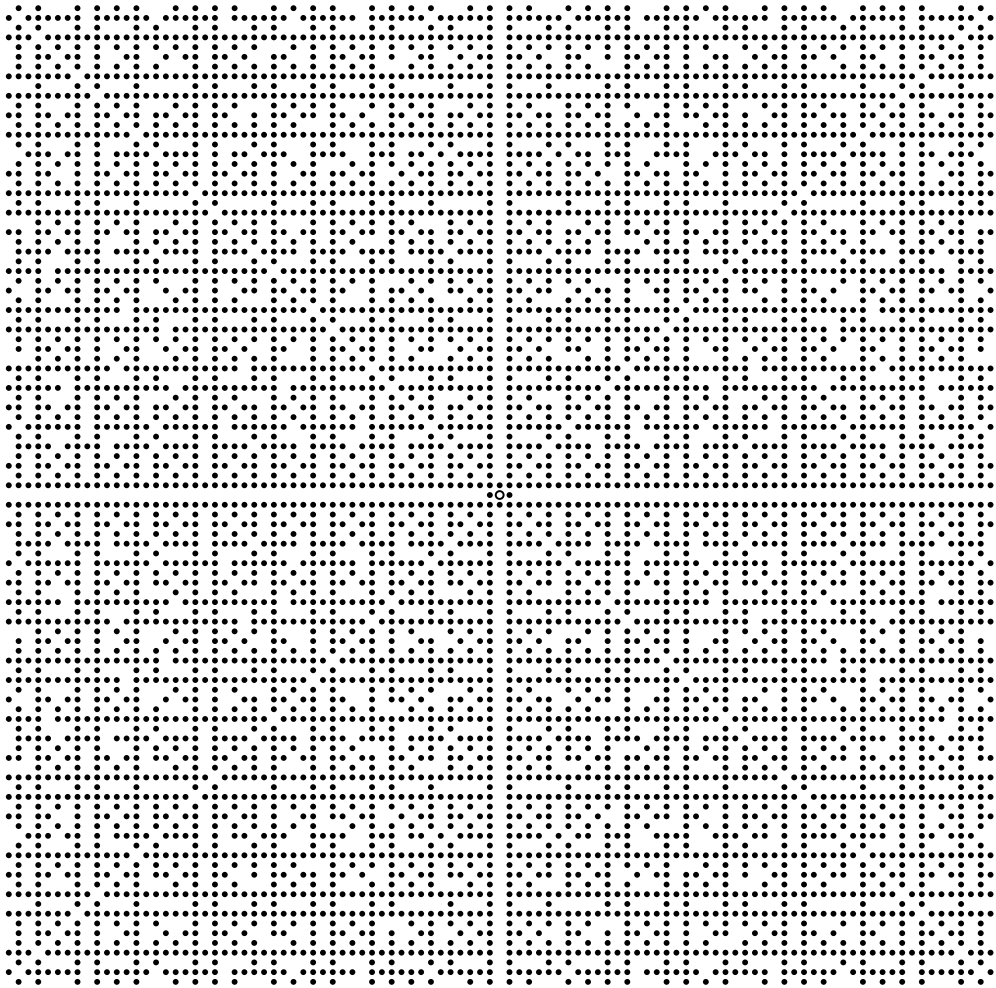}
\caption{A central patch of the visible points $V$ of the square lattice
  $\mathbb Z^2$. Note the invariance of $V$ with respect to $\operatorname{GL}(2,\mathbb Z)$.}
\label{fig: visible}
\end{figure}
\end{center}

Note that big holes are rare, but important; see~\cite[Rem.\
10.6]{TAO} for some examples. The interesting fact is
that they do not destroy the existence of patch frequencies, though
the latter clearly cannot exist uniformly. For the natural pair correlation
(or autocorrelation) coefficients,
$$
\eta(x)\,:=\,\lim_{R\to\infty}\frac{1}{\pi
  R^2}\big |V\cap (-x+V)\cap B_R(0)\big |,
$$
one finds the following result; compare~\cite[Lemma 10.6]{TAO},
\cite[Thm.\ 2]{BMP}, as well as \cite[Thm.\ 7]{PH}.

\begin{lemma}
For each\/ $x\in\mathbb Z^2$, the natural autocorrelation coefficient
$\eta(x)$ of\/ $V$ exists, and is given by
$$
\eta(x)\,=\,\xi\!\prod_{p\mid\operatorname{gcd}(x)}\left(1+\frac{1}{p^2-2} \right),
$$
where\/ $\xi=\prod_p(1-2p^{-2})\approx 0.3226$. In particular, with
$\operatorname{gcd}(0)=0$, this also gives the density\/ $\eta(0)=\prod_p(1-p^{-2})=1/\zeta(2)$.\qed
\end{lemma}

The autocorrelation measure of $V$, 
$$
\gamma\,=\,\sum_{x\in\mathbb Z^2}\eta(x)\delta_x,
$$
is thus well-defined, and is a translation bounded, positive definite measure on $\mathbb R^2$
by construction. Its Fourier transform $\widehat\gamma$ exists by
general arguments, compare~\cite[Ch.\ I.4]{BF} or~\cite[Rem.~8.7 and Prop.\ 8.6]{TAO}, and leads to the following result;
see~\cite[Thm.\ 10.5]{TAO} and~\cite[Thm.\ 3]{BMP}.

\begin{center}
\begin{figure}
\includegraphics[width=0.84\textwidth]{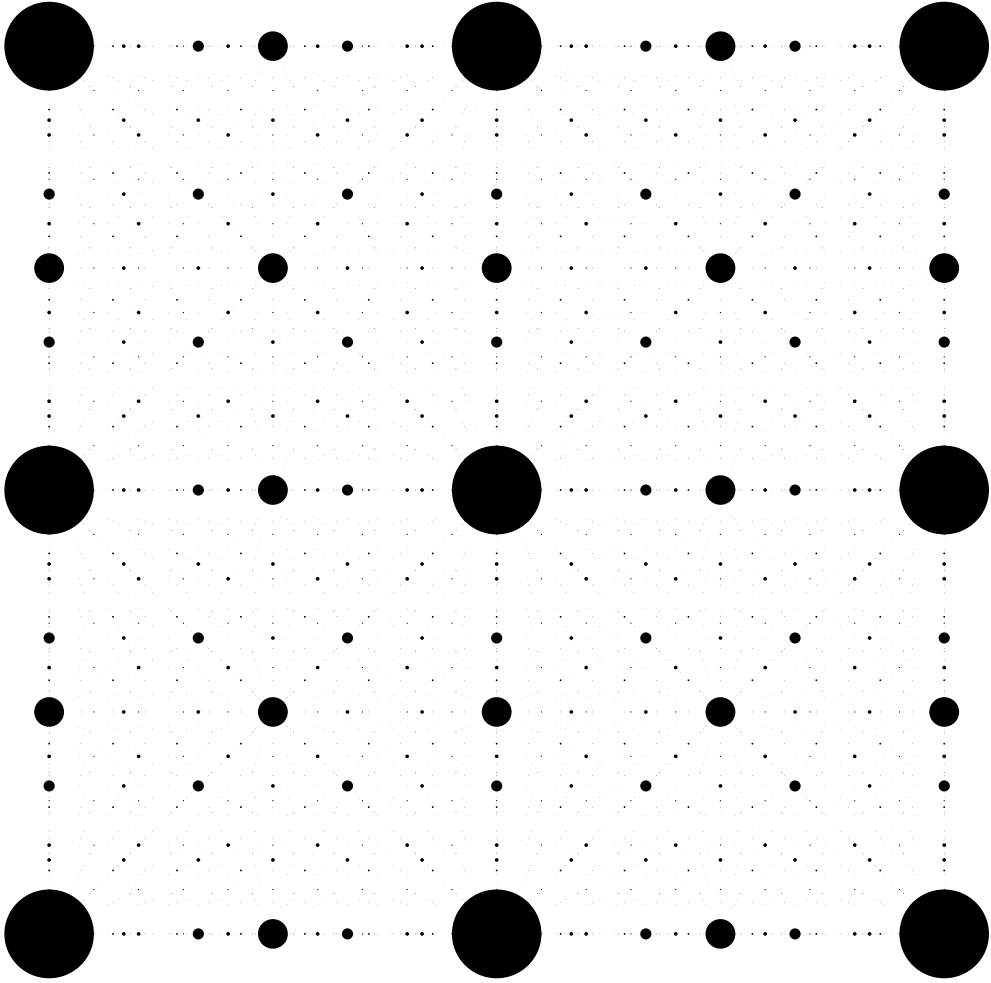}
\caption{Diffraction $\widehat{\gamma}$ of the visible points of
  $\mathbb Z^2$. A point measure at $k$ with intensity $I(k)$ is shown
  as a disk centred at $k$ with area proportional to $I(k)$. Shown are
the intensities with $I(k)/I(0)\ge 10^{-6}$ and $k\in [0,2]^2$. Its lattice of periods is $\mathbb Z^2$, and $\widehat{\gamma}$ turns out to
be $\operatorname{GL}(2,\mathbb Z)$-invariant.}
\label{fig: vispodiff}
\end{figure}
\end{center}

\begin{theorem}\label{diff}
The natural diffraction measure\/ $\widehat \gamma$ of the visible points\/ $V$ of
the square lattice\/ $\mathbb Z^2$ exists. It is a positive pure point
measure which is translation bounded and supported on the points of\/ $\mathbb Q^2$ with
square-free denominator, the Fourier-Bohr spectrum of $\gamma$, so
$$
\widehat{\gamma}=\sum_{{\substack{k\in\mathbb
      Q^2\\\operatorname{den}(k) \text{\scriptsize \,square-free}}}}I(k)\delta_k,
$$
where\/ $\operatorname{den}(k):=\operatorname{gcd}\{n\in\mathbb
N\,|\,nk\in\mathbb Z^2\}$. In particular,
$I(0)=(1/\zeta(2))^2=36/\pi^4$, and when\/ $0\neq k\in\mathbb Q^2$ has
square-free denominator
$\operatorname{den}(k)$, the corresponding intensity is given by
$$
I(k)=\left(\frac{6}{\pi^2}\prod_{p\mid\operatorname{den}(k)}\frac{1}{p^2-1}\right)^2.
$$
\qed
\end{theorem}

The essence of this result is the pure point nature of
$\widehat\gamma$ together with its explicit computability via an
intensity formula in the form of a {\em finite}\/ product for any given
$k\in\mathbb Q^2$ with square-free denominator. Fig.~\ref{fig:
  vispodiff} illustrates the diffraction measure. Note that
$\widehat\gamma$ has the symmetry group $\mathbb Z^2\rtimes
\operatorname{GL}(2,\mathbb Z)$. 

An alternative view is possible by means of the Herglotz--Bochner
theorem as follows. The autocorrelation measure $\gamma$ is positive
definite on $\mathbb R^2$ if and only if the function $\eta\!:\, \mathbb
Z^2\longrightarrow \mathbb R$ is positive definite on
$\mathbb Z^2$; see~\cite[Lemma 8.4]{TAO}. The latter property is
equivalent to the existence of a positive measure $\varrho$ on $\mathbb
T^2=\mathbb R^2/\mathbb Z^2\simeq [0,1)^2$ such that
$$
\eta(x)\,=\int_{\mathbb
T^2} e^{2\pi i xy}\,\,{\rm d}\varrho(y),
$$
where the connection to $\widehat\gamma$ is established by
$\varrho=\gamma\!\mid_{[0,1)^2}$, so that $\widehat\gamma=\varrho\ts * \delta_{\mathbb
  Z^2}$, where $\delta_{S} := \sum_{x\in S} \delta_{x}$ denotes
   the Dirac comb of a discrete point set $S$. The finite positive measure $\varrho$ is a spectral measure
in the sense of dynamical system theory. It is related to the
diffraction measure by convolution; for background, we refer
to~\cite{BL,BLvE} and references therein. We shall return to the
dynamical point of view shortly.\medskip

Let us pause to comment on the history and the development of this problem. The arithmetic
properties of $V$ are classic and can be found in many
places, including~\cite{Apostol,Hua}. The investigation of spectral
aspects was advertised by Schroeder in~\cite{Schroeder1}, see also~\cite{Schroeder2}, by means of a numerical
approach via FFT techniques. These results suffered from insufficient
resolution in the numerical treatment, and seemed to point towards
continuous diffraction components, perhaps in line with the idea that
the distribution of primes is sufficiently `random'.

Ten years later, on the basis of a formal M\"obius inversion
calculation for the amplitudes, Mosseri argued in~\cite{Mosseri} that the diffraction should be pure point rather than
continuous, thus contradicting the earlier numerical findings of
Schroeder. This was corroborated in~\cite{BGW} with
further calculations on the diffraction intensities (still without
proof), which gave the formula of Thm.~\ref{diff} above. Also, a
rather convincing comparison with an optical diffraction experiment
was shown, which clearly indicated the correctness of the formal
calculation. The first complete proof, with a detailed convergence
result with precise error estimates, appeared in [BMP], and was
recently improved and extended in~\cite{PH}, on the basis of
number-theoretic results
due to Mirsky~\cite{Mirsky1,Mirsky2}.

Simultaneously, due to the renewed interest in the square-free integers
in connection with Sarnak's conjecture, the dynamical sytems point of
view became more important, as is obvious
from~\cite{CS,CV,PH,HB}. Here, the focus is more on the dynamical
spectrum, which is closely related to the diffraction
measure as indicated above, and explained in detail in~\cite{BL,BLvE}.

To explain this, let us define the (discrete) hull of $V$ as
$$
\mathbb X_{V}=\overline{\{t+V\,|\,t\in\mathbb Z^2\}},
$$ 
where the closure is taken in the product topology induced on
$\{0,1\}^{\mathbb Z^2}$ by the discrete topology on $\{0,1\}$. This
topology is metric~\cite{Sol,PH} and is 
also called the \emph{local topology}, because two elements of
$\{0,1\}^{\mathbb Z^2}$ are close if they agree on a large ball around
the origin. Clearly, $\mathbb X_{V}$ is then compact, where
here and below we simultaneously view subsets of $\mathbb Z^2$ as configurations. In
particular, the empty set is identified with the configuration
$\underline 0$, and $\mathbb Z^2$ with $\underline 1$ this way. Since
$V$  contains holes of arbitrary size, the empty set is an element of
$\mathbb X_{V}$.

For
a natural number $m$, let $\cdot\,_m\!:\,\mathbb
Z^2\longrightarrow \mathbb
Z^2/m\mathbb
Z^2$ denote the canonical projection $x\mapsto  [x]_m$,
where $[x]_m=x+m\mathbb Z^2$. For a subset $X\subset \mathbb Z^2$, we
denote by $X_m$ its image under this projection map. It
should however be born in mind that, for
elements $x\in\mathbb Z^2$, their images under the above map will always be written as $[x]_m$
rather than $x_m$. This is due to the occasional need of regarding $m$
in the last expression as an
index. Let us recall the
following result from~\cite{BMP}.

\begin{prop}[Chinese Remainder Theorem]{\rm \cite[Prop.\ 2]{BMP}}\label{crt}
For pairwise coprime positive integers\/ $m_1,m_2,\ldots,m_r$, the
natural group homomorphism
$$
(\mathbb Z^2)_{m_1m_2\cdot\ldots \cdot m_r} \,\,\longrightarrow \,\,\prod_{i=1}^r
(\mathbb Z^2)_{m_i}$$
is an isomorphism.  In particular, for\/ $x_1,x_2,\ldots,x_r\in\mathbb
Z^2$, the simultaneous solutions\/ $t\in\mathbb Z^2$ of
$$
[t]_{m_i}= [x_i]_{m_i}, \quad 1\le i\le r,
$$
comprise precisely one coset of\/ $m_1m_2\cdot\ldots\cdot  m_r\mathbb Z^2$ in
$\mathbb Z^2$. \qed
\end{prop}

Next, let us come to a characterisation of $\mathbb X_{V}$. Let
$\mathbb A$ denote 
the set of \emph{admissible} subsets $A$
of $\mathbb Z^2$, i.e.\ subsets $A\subset\mathbb Z^2$ with the property that, 
for every prime $p$, $A$ does \emph{not} contain a full set of
representatives modulo $p\mathbb Z^2$. In other words, $A$ is
admissible if and only if 
$$|A_p|<p^2$$ for any prime $p$. Since $V\in\mathbb A$ (otherwise some point
of $V$ would be in $p\mathbb Z^2$
for some prime $p$, a contradiction) and since $\mathbb A$ is a
$\mathbb Z^2$-invariant and
closed subset of $\{0,1\}^{\mathbb Z^2}$, it is clear that $\mathbb
X_{V}$
is a subset of $\mathbb A$. By~\cite[Thm.~2]{PH}, the other inclusion is also
true. This was first shown by Herzog and Stewart~\cite{HS} for visible
lattice points and by Sarnak~\cite{Sarnak} for the analogous case of
the square-free integers. In fact, similar statements hold true for
various generalisations discussed below; cf.~\cite[Thm.~6]{PH} for the
case of $k$-free lattice points.

\begin{theorem}\label{charachull}
One has\/ $\mathbb X_{V}=\mathbb A$.\qed
\end{theorem} 

It follows that $\mathbb X_{V}$ is \emph{hereditary}, i.e.\
$$
\forall\, X\in\mathbb A:\,(Y\subset X\Rightarrow
Y\in\mathbb A),
$$
and in particular contains
\emph{all} subsets of $V$. In other
words, $V$
is an \emph{interpolating set}\/ for $\mathbb X_{V}$ in the sense
of~\cite{W}, which means that $$\mathbb X_{V}|^{}_V\,\,:=\{X\cap
V\,\mid\,
X\in\mathbb X_{V}\}=\{0,1\}^{V}.$$

Given a radius $\rho>0$ and a point $t\in\mathbb Z^2$,
the \emph{$\rho$-patch} of $V$ at $t$ is
\[(V-t)\cap B_\rho(0),\]
the translation to the origin of the part of $V$ within a distance
$\rho$ of $t$. We denote by $\mathcal A(\rho)$ the (finite) set of all
$\rho$-patches of $V$, and by $N(\rho)=|\mathcal A(\rho)|$ the number of
distinct $\rho$-patches of $V$. For a $\rho$-patch $\mathcal P$ of
$V$, denote by $C_{\mathcal P}$ the set of elements of
$\mathbb X_{V}$ whose $\rho$-patch at
$0$ is $\mathcal P$, the so-called \emph{cylinder set} defined by the
$\rho$-patch $\mathcal P$; compare~\cite{Denker}. Note that these cylinder sets form a
basis of the topology of $\mathbb X_{V}$. 

The \emph{patch counting entropy} of
$V$ is defined as
$$
h_{\rm pc}(V):=\lim_{\rho\to\infty}\frac{\log N(\rho)}{\pi\rho^2}.
$$
Note that this differs from the definition in~\cite{PH,HB}, where, in
view of the binary configuration space interpretation, we
used the base $2$ logarithm. It can be shown by a classic subadditivity argument that this limit
exists. Since $\mathbb X_{V}$ is hereditary, it follows that $V$ has patch counting entropy $h_{\rm
  pc}(V)$ at least
$\operatorname{dens}(V)\log (2)=(6/\pi^2)\log (2)$. In fact, one has more.

\begin{theorem}{\rm \cite[Thm.~3]{PH}}\label{hpc}
One has\/ $h_{\rm pc}(V)=(6/\pi^2)\log (2)$. \qed
\end{theorem} 

The natural translational
action of the group $\mathbb Z^2$
on $\mathbb X_{V}$ is continuous and $(\mathbb
X_{V},\mathbb Z^2)$ thus is a \emph{topological dynamical
  system}. By construction, $(\mathbb X_{V},\mathbb Z^2)$ is topologically
transitive~\cite{A,G,W}, as it is the orbit closure of one of its
elements (namely $V$). Equivalently, for any two non-empty open subsets $U$
and $W$ of $\mathbb X_{V}$, there is an element $t\in\mathbb Z^2$ such
that
$$
U\cap (W+t)\neq\varnothing.
$$
In accordance with Sarnak's findings~\cite{Sarnak} for
square-free integers, one has the following result.

\begin{theorem}\label{c1}
The topological dynamical system\/ $(\mathbb X_{V},\mathbb Z^2)$ has the following properties.
\begin{itemize}
\item[\rm (a)]
$(\mathbb X_{V},\mathbb Z^2)$ is topologically ergodic with positive topological 
entropy equal to\/ $(6/\pi^2)\log (2)$.
\item[\rm (b)]
$(\mathbb X_{V},\mathbb Z^2)$ is proximal, and\/ $\{\varnothing\}$ is
the unique\/ $\mathbb Z^2$-minimal subset of\/ $\mathbb X_{V}$.
\item[\rm (c)]
$(\mathbb X_{V},\mathbb Z^2)$ has no non-trivial topological Kronecker
factor\/ $($i.e., minimal equicontinuous factor\/$)$. In particular,\/ $(\mathbb
X_{V},\mathbb Z^2)$ has trivial topological point spectrum.
\item[\rm (d)]
$(\mathbb X_{V},\mathbb Z^2)$ has a non-trivial joining with the
Kronecker system given by\/ $K=(G,\mathbb Z^2)$, where\/ $G$ is the compact
Abelian group\/ $\prod_p (\mathbb Z^2)_p$ and\/ $\mathbb Z^2$ acts on\/ $G$
via addition of $\iota(x)=([x]_p)$, i.e.\  $g\mapsto
g+\iota(x)$, with\/ $g\in G$ and $x\in\mathbb Z^2$. In
particular,\/ $(\mathbb X_{V},\mathbb Z^2)$ fails to be topologically weakly mixing.
\end{itemize}
\end{theorem}
\begin{proof}
The topological entropy of the dynamical system
$(\mathbb X_{V},\mathbb Z^2)$ is just $h_{\rm pc}(V)$, so the assertion
follows from Theorem~\ref{hpc}; cf.~\cite[Thm.~1]{BLR}. 
The topological ergodicity~\cite{A,G} will follow from the existence
of an ergodic full (non-empty open subsets have positive measure) $\mathbb Z^2$-invariant Borel measure on $\mathbb X_{V}$;
see~Theorems~\ref{freq} and~\ref{c2}(b) below.

For part (b), recall from Theorem~\ref{charachull} that the hull
contains many more elements than the translates of $V$. Nevertheless, one can derive from
Proposition~\ref{propbasic}(a) that every element of $\mathbb
X_{V}$ contains holes of arbitrary size
that repeat lattice-periodically. This follows by standard compactness
arguments from considering a
sequence of the form $(t_n+V)^{}_{n\in\mathbb N}$ that converges in the local
topology, via selecting suitable subsequences. In particular, let
$X,Y\in\mathbb X_V$ and a radius $\rho$ be fixed. Let $t_\rho+\varGamma$ be
 positions of holes of inradius $\rho$ in $X$. Choose $\rho'$ large
 enough such that $B_{\rho'}(0)$ covers $B_{\rho}(0)+F$, where $F$ is
 a fundamental domain of $\varGamma$. Then, any $\rho'$-hole of $Y$
 (which exists) contains a $\rho$-hole of $X$. Hence, for any $\rho>0$
and any two
elements $X,Y\in\mathbb
X_{V}$, there is a translation $t\in\mathbb Z^2$ such that
$$(X+t)\cap B_\rho(0)=(Y+t)\cap B_\rho(0)=\varnothing,$$ meaning that both $X$ and $Y$
have the empty $\rho$-patch at $-t$. In terms of the metric $d$ on
$\mathbb X_V$~\cite{Sol,PH,HB} this means
that $d(X+t,Y+t)\le
1/\rho$ and the proximality of the system follows. Similarly, the assertion on the unique $\mathbb Z^2$-minimal
subset $\{\varnothing\}$ follows from the fact that any element of $\mathbb
X_{V}$ contains arbitrarily large holes and thus any non-empty closed 
subsytem contains $\varnothing$. 

Since Kronecker systems are distal, the first assertion of part (c) is an immediate consequence of the
proximality of $(\mathbb X_{V},\mathbb Z^2)$. This also 
implies that $(\mathbb X_{V},\mathbb Z^2)$ has trivial
topological point spectrum; see~\cite{HB} for an alternative argument
that the non-zero constant function is the only continuous eigenfunction of the
translation action.

For part (d), one can verify that a non-trivial (topological) joining~\cite{G} 
of $(\mathbb X_{V},\mathbb Z^2)$ with the Kronecker system $K$ is given
by 
$$
W:=\bigcup_{X\in\mathbb X_{V}}\Big(\{X\}\times \prod_p
\bigl(\mathbb Z^2\setminus X\bigr)_p\Big).
$$
Since the Kronecker system $K$ is minimal and distal, a well-known
disjointness theorem by Furstenberg~\cite[Thm.~II.3]{F} implies that
$(\mathbb X_{V},\mathbb Z^2)$ fails to be topologically weakly mixing.
\end{proof}

Following~\cite{BMP,PH}, the natural \emph{frequency} $\nu(\mathcal P)$
of a $\rho$-patch $\mathcal P$ of $V$ is defined as
\begin{equation}\label{freqdef}
\nu(\mathcal
P):=\operatorname{dens}\big(\{t\in\mathbb Z^2\,\mid\,(V-t)\cap B_\rho(0)=\mathcal P\}\big),
\end{equation}
which can indeed be seen to exist. 

\begin{theorem}{\rm \cite[Thms.~1 and~2]{PH}}\label{freq}
Any\/ $\rho$-patch\/ $\mathcal P$ of\/ $V$ occurs with positive
frequency, which is given by
\[\nu(\mathcal P)=\sum_{\mathcal F\subset (\mathbb Z^2\cap B_{\rho}(0))\setminus \mathcal P}(-1)^{|\mathcal F|}
\prod_p\left(1-\frac{|(\mathcal P\cup\mathcal
    F)_p|}{p^{2}}\right).\]
\qed
\end{theorem}

The frequency function $\nu$ from~\eqref{freqdef}, regarded as a function on the
cylinder sets by setting $\nu(C_\mathcal
P):=\nu(\mathcal P)$, is finitely additive on the
cylinder sets with $$\sum_{\mathcal P\in\mathcal
  A(\rho)}\nu(C_{\mathcal P})=\frac{1}{|\det(\mathbb Z^2)|}=1.$$ Since the
family of cylinder sets is a (countable) semi-algebra that generates the
Borel $\sigma$-algebra on $\mathbb X_{V}$ (i.e.\ the smallest 
$\sigma$-algebra on $\mathbb X_{V}$ which contains the open subsets of
$\mathbb X_{V}$), $\nu$ extends uniquely to a probability measure on
$\mathbb X_{V}$; compare~\cite[Prop. 8.2]{Denker} and references given
there. Moreover, this probability measure is
$\mathbb Z^2$-invariant by construction wherefore we have a
measure-theoretic dynamical system $(\mathbb X_{V},\mathbb Z^2,\nu)$. For part (b) of the following
claim, note that 
the Fourier--Bohr spectrum of $V$ is itself a group and
compare~\cite[Prop. 17]{BLvE}.

\begin{theorem} \label{c2}
The measure-theoretic
dynamical system\/ $(\mathbb X_{V},\mathbb Z^2,\nu)$ has the following properties.
\begin{itemize}
\item[\rm (a)]
The\/ $\mathbb Z^2$-orbit of\/ $V$ in\/ $\mathbb X_{V}$ is\/ $\nu$-equidistributed,
which means that for any function\/ $f\in C(\mathbb X_{V})$, one has
\[
\lim_{R\to\infty}\frac{1}{\pi R^2}\sum_{x\in\mathbb Z^2\cap
  B_R(0)}f(V+x)=\int_{\mathbb X_{V}}f(X)\,\,{\rm d}\nu(X).
\]
In other words,\/ $V$ is\/ $\nu$-generic.
\item[\rm (b)]
$(\mathbb X_{V},\mathbb Z^2,\nu)$ is ergodic, deterministic
$($i.e., it is of zero measure entropy\/$)$ and has pure point
dynamical spectrum. The latter is given by
the Fourier--Bohr spectrum of the autocorrelation\/ $\gamma$, as
described in Theorem~$\ref{diff}$.
\item[\rm (c)]
$(\mathbb X_{V},\mathbb Z^2,\nu)$ is metrically 
isomorphic to the Kronecker system\/ $K_{\mu}=(G,\mathbb Z^2,\mu)$, where\/ $G$ is the
compact Abelian 
group\/ $\prod_p (\mathbb Z^2)_p$, the lattice\/ $\mathbb Z^2$ acts on\/ $G$
 as
above and
$\mu$ is the normalised 
Haar measure on\/ $G$.
\end{itemize}
\end{theorem}
\begin{proof}
For part (a), it suffices to show this for the characteristic
functions of cylinder sets of finite patches, as their span is dense
in $C(\mathbb X_{V})$. But for such functions, the claim is clear as
the left hand side is the patch frequency as used for the definition
of the measure $\nu$.

For the ergodicity of $(\mathbb X_{V},\mathbb Z^2,\nu)$, one has to show
that
$$
\lim_{R\rightarrow\infty}\frac{1}{\pi R^2}\sum_{x\in\mathbb Z^2\cap
  B_R(0)}\nu\big((C_\mathcal P+x)\cap C_\mathcal Q\big)=\nu(C_\mathcal P)\nu(C_\mathcal Q)
$$
for arbitrary cylinder sets $C_\mathcal P$ and $C_\mathcal Q$;
compare~\cite[Thm.~1.17]{Walters}. The latter in turn follows from a
straightforward (but lengthy) calculation using
Theorem~\ref{freq} and the definition of the measure $\nu$ together
with the Chinese Remainder Theorem. In fact, for technical
reasons, it is better to work with a different semi-algebra that also 
generates the Borel $\sigma$-algebra on $\mathbb X_{V}$; see Appendix~\ref{appa} for the
details.

Vanishing measure-theoretical entropy, i.e.\
$$
h_{\rm meas}(V)=\lim_{\rho\to\infty}
\frac{1}{\pi\rho^2}\sum_{\mathcal P\in\mathcal
  A(\rho)}\!\!\!-\nu(C_\mathcal P)\log \nu(C_\mathcal P)=0,
$$
was shown
in~\cite[Thm.~4]{PH}, which is in line with the results
of~\cite{BLR}. Alternatively, it is an immediate consequence of part
(c) above. As a consequence of part (a), the individual
diffraction measure of $V$ according to Theorem~\ref{diff} coincides
with the diffraction measure of the system $(\mathbb
X_V,\mathbb Z^2,\nu)$ in the sense of~\cite{BL}. Then, pure point
diffraction means pure point dynamical spectrum~\cite[Thm.~7]{BL},
and the latter is the group generated by the Fourier--Bohr spectrum;
compare~\cite[Thm.~8]{BL} and~\cite[Prop. 17]{BLvE}. Since the intensity
formula of Theorem~\ref{diff} shows that there are no extinctions,
the Fourier--Bohr spectrum here is itself a group, which completes
part (b).

It is well known that $K_{\mu}=(G,\mathbb Z^2,\mu)$ has the same pure
point 
dynamical spectrum as $(\mathbb
X_V,\mathbb Z^2,\nu)$; compare~\cite{CS} for the details in the case
of the square-free integers. In particular, the subgroup of $\mathbb
T^2$ given by the points of
$\mathbb Q^2\cap [0,1)^2$ with square-free denominator is easily seen to be isomorphic
to the direct sum $\bigoplus_p \mathbb Z^2/p\mathbb Z^2$, wherefore it is
the Pontryagin dual of the
direct product 
$G=\prod_p \mathbb Z^2/p\mathbb Z^2$; cf.~\cite[Sec.\ 2.2]{Rudin}. By a theorem of von Neumann~\cite{vNeumann}, two ergodic measure-preserving transformations
with pure point dynamical spectrum are isomomorphic if and only if
they have the same dynamical spectrum. This proves part
(c), which is a particular instance of the Halmos--von Neumann
theorem; cf.~\cite{HvN}.

Alternatively, the Kronecker system can be read off from the model set
description, which also provides the compact Abelian group. The general formalism is developed in~\cite{BLM}, though
the torus parametrisation does not immediately apply. Some extra work
is required here to establish the precise properties of the
measure-theoretic 
homomorphism onto the compact Abelian group. Diagrammatically, 
the construction looks like this:
\[\begin{array}{ccccc}
\mathbb Z^2&\longleftarrow&\mathbb Z^2\times \prod_p (\mathbb Z^2)_p &\longrightarrow&\prod_p (\mathbb Z^2)_p\\
\cup&&\cup&&\cup\\ M&&L&&W\\ \uparrow&&\uparrow&&\uparrow\\
x&\longleftrightarrow&(x,\iota(x))&\longrightarrow&\iota(x)
\end{array}\]
Here $$L:=\big\{(x,\iota(x))\,\big|\,x\in\mathbb Z^2\big\}=\big\{(x,([x]_p))\,\big|\,x\in\mathbb Z^2\big\}$$ is the natural (diagonal) embedding of $\mathbb Z^2$ into $\mathbb
Z^2\times \prod_p (\mathbb Z^2)_p$ and 
$$
W:=\prod_p \big((\mathbb Z^2)_p\setminus\{[0]_p\}\big)
$$
satisfies $W=\partial W$ and has measure $\mu(W)=\prod_p(1-\tfrac{1}{p^2})=1/\zeta(2)$ with respect to the
normalised Haar measure $\mu$ on the compact group $\prod_p (\mathbb
Z^2)_p$. Clearly, one has
$$
M:=M(W):=\{x\in\mathbb Z^2\,|\,\iota(x)\in W\}=V
$$
The above diagram is in fact a \emph{cut and project scheme}\/: $L$ is a
lattice (a discrete co-compact subgroup) in $\mathbb Z^2\times \prod_p (\mathbb Z^2)_p$
 with one-to-one projection onto the first
factor and dense projection onto the second factor. This means that $V$ is a \emph{weak model
set}~\cite{TAO}. The corresponding `torus' is
$$
\mathbb T\,:=\,\big(\mathbb Z^2\times \prod_p (\mathbb Z^2)_p\big)\big/L\,\,\simeq \,\,\prod_p(\mathbb Z^2)_p,
$$
with the $\mathbb Z^2$-action given by addition of
$\iota(x)=([x]_p)$. A similar construction with translations from the group
$\mathbb R^2$ in mind is
given in~\cite{BMP}; see also~\cite[Ch.\ 5a]{Sing}.

The so-called \emph{torus parametrisation}~\cite{Moody} is the Borel map
$$
\varphi\!:\, \mathbb T\rightarrow \mathbb X_V,
$$  
given by 
$$
([y_p]_p)\mapsto M\big(([y_p]_p)+W\big).
$$
Clearly, $\varphi$ intertwines the $\mathbb Z^2$-actions. Note that, since $\mathbb X_V=\mathbb A$, one indeed has
$$
M(([y_p]_p)+W)=\mathbb Z^2\setminus\bigcup_p (y_p+p\mathbb Z^2)\in\mathbb X_V.
$$
The map $\varphi$ fails to be injective. For example, the fibre $\varphi^{-1}(\varnothing)$ over
the empty set $\varnothing$ is easily seen to be
uncountable. Furthermore, $\varphi$ is not continuous, since, e.g.,
$\varphi(([(0,0)]_p))=V$ but $$V\owns (1,0)\not\in\varphi\big(([(0,0)]_{p_1},\dots,[(0,0)]_{p_{n-1}},[(1,0)]_{p_n},[(0,0)]_{p_{n+1}},\dots)\big).$$ Nevertheless, 
employing the ergodicity of the measure $\nu$, one can show that
$\varphi$ is in fact a measure-theoretic isomorphism between the two systems; see
Appendix~\ref{appb} for details.
\end{proof}

Let us mention that our approach is complementary to that
of~\cite{CS}. There, ergodicity and pure point spectrum are consequences of
determining all eigenfunctions, then concluding via $1$ being a simple
eigenvalue and via the basis property of the eigenfunctions. Here, we
establish the $\nu$-genericity of $V$ and the ergodicity of the measure $\nu$ and afterwards use the equivalence
theorem between pure point dynamical and diffraction spectrum~\cite[Thm.~7]{BL},
hence employing the diffraction measure of $V$ calculated in~\cite{BMP,PH}.

\section{$k$-free lattice points}\label{kfree}

The square-free integers and the visible points of the square lattice are
particular cases of the following natural generalisation. Let
$\varLambda\subset\mathbb R^n$ be a lattice. The \emph{$k$-free
  points} $V=V(\varLambda,k)$ of $\varLambda$ are then defined by 
$$
V\,=\,\mathbb
\varLambda\setminus\bigcup_{\text{$p$ \scriptsize prime}} p^k\varLambda.
$$
They are the points with the property that the
greatest common divisor of their 
coordinates (in an arbitrary lattice basis) is not divisible by any non-trivial
$k$th power of an integer. Without restriction, we shall assume that
$\varLambda$ is
\emph{unimodular}, i.e.\ $|\det(\varLambda)|=1$. Moreover, we exclude the trivial
case $n=k=1$, where $V$ consists of just the two points of $\varLambda$ closest to $0$ on either side. On the basis of the results
in~\cite{BMP,PH}, one can then show analogous versions of any of the above
findings. In particular, one has~\cite[Cor.~1]{PH}
$$
\operatorname{dens}(V)=\frac{1}{\zeta(nk)}
$$
and the result for the diffraction measure $\widehat\gamma$ of $V$ looks
as follows. Recall that the \emph{dual}\/ or \emph{reciprocal 
lattice}\/ $\varLambda^*$ of $\varLambda$ is
\[ 
\varLambda^*:=\{y \in\mathbb{R}^n\,\mid\, y\cdot x\in\mathbb Z
\mbox{ for all } x\in\varLambda\}.
\]
Further, the
\emph{denominator} of a point $\ell$ in the $\mathbb Q$-span $\mathbb
Q\varLambda^*$ of $\varLambda^*$ is defined as
$$
\operatorname{den}(\ell):=\operatorname{gcd}\{m\in\mathbb N\,\mid\,m \ell\in\varLambda^*\}.
$$

\begin{theorem}\cite[Thms.~3 and 5]{BMP} \cite[Thm.~8]{PH}\label{thdiff}
The natural diffraction measure $\widehat{\gamma}$ of the autocorrelation
$\gamma$ of\/ $V$ exists. It is a positive pure point measure which is
translation bounded and supported on the set of points in $\mathbb Q\varLambda^*$
with $(k+1)$-free denominator, so
$$
\widehat{\gamma}=\sum_{{\substack{\ell\in\mathbb
      Q\varLambda^*\\\text{\scriptsize $\operatorname{den}(\ell)$  $(k+1)$-free}}}}I(\ell)\,\delta_l.
$$
In particular,
$I(0)=(1/\zeta(nk))^2$ and when\/ $0\neq \ell\in\mathbb Q\varLambda^*$ has
$(k+1)$-free denominator
$\operatorname{den}(\ell)$, the corresponding intensity is given by 
$$
I(\ell)=\Bigg(\frac{1}{\zeta(nk)}\prod_{p\mid \operatorname{den}(\ell)}\frac{1}{p^{nk}-1}\Bigg)^2.
$$
\qed
\end{theorem}
Again, the hull $$\mathbb X_V=\overline{\{t+V\,|\,t\in\varLambda\}}$$ of $V$ turns out to be just the set of
\emph{admissible}\/ subsets of $\varLambda$, i.e.\ subsets $A$ of $\varLambda$
with
$$
|A_p|<p^{nk}
$$
for any prime $p$, where $A_p$ denotes the reduction of $A$ modulo
$p^k\varLambda$; see~\cite[Thm.~6]{PH}. The natural topological dynamical system $(\mathbb
X_V,\varLambda)$ has the analogous properties as in the special case
discussed above in Theorem~\ref{c1}. In particular, it has positive topological entropy
equal to the patch counting entropy of $V$, i.e.\
$$h_\T(V)=\frac{\log (2)}{\zeta(nk)}$$ 
by~\cite[Thm.~3]{PH}. For the patch frequencies, one has the following
result.
 \begin{theorem}{\rm \cite[Thms.~1 and~2]{PH}}\label{freq2}
Any\/ $\rho$-patch\/ $\mathcal P$ of\/ $V$ occurs with positive
frequency, which is given by
\[\nu(\mathcal P)=\sum_{\mathcal F\subset (\varLambda\cap B_{\rho}(0))\setminus \mathcal P}(-1)^{|\mathcal F|}
\prod_p\left(1-\frac{|(\mathcal P\cup\mathcal
    F)_p|}{p^{nk}}\right).\]
\qed
\end{theorem}

This gives rise to a measure-theoretic dynamical
system $(\mathbb X_{V},\varLambda,\nu)$ which can be seen, as above,
to be ergodic and metrically isomorphic to $(G,\varLambda,\mu)$, where $G$ is the
compact Abelian group
$$
G=\prod_p\varLambda_p=\prod_p\varLambda/p^k\varLambda
$$
on which the lattice $\varLambda$ acts via addition of
$\iota(x)=([x]_p)$. As before, $\mu$ is the normalised Haar measure on $G$. It follows that $(\mathbb X_{V},\varLambda,\nu)$ has
zero measure entropy; see also~\cite[Thm.~4]{PH}. Again, $V$ turns out
to be $\nu$-generic. We thus get the analogous result to
Theorem~\ref{c2} also in this more general setting.

\section{$\mathscr B$-free lattice points}\label{bfree}

One further generalisation step seems possible as follows. In~\cite{ELD}, 
Lema\'nczyk et al.\ studied the dynamical properties of $\mathscr B$-free systems, i.e.\
$\mathscr B\subset\{2,3,\dots\}$ consists of
pairwise coprime integers satisfying
$$
\sum_{b\in\mathscr B}\frac{1}{b}<\infty
$$
and the hull $\mathbb X_{\mathscr
  B}=\overline{\{t+V\,|\,t\in\mathbb Z\}}$ is the orbit closure of
the set
$$
V\,=\,\mathbb Z\setminus\bigcup_{b\in\mathscr B} b\mathbb Z
$$
of \emph{$\mathscr B$-free numbers}\/ (integers with no factor from
$\mathscr B$). Replacing the one-dimensional lattice $\mathbb Z\subset \mathbb
R$ by other unimodular 
lattices $\varLambda\subset\mathbb R^n$ in the above definitions and
requiring that $\sum_{b\in\mathscr B}\frac{1}{b^n}<\infty$, one arrives
at  \emph{$\mathscr B$-free lattice points}\/ $V$ and the associated
topological dynamical systems $(\mathbb X_V,\varLambda)$. The $k$-free lattice points from the previous
section then arise from the particular choice $\mathscr B=\{p^k\,|\,\text{$p$ prime}\}$. Since the proofs in~\cite{BMP,PH} do not use
special properties of $k$th powers of prime numbers except their
pairwise coprimality, the above results carry over to the case of $\mathscr
B$-free lattice points with almost identical proofs. In particular,
the density and topological (patch counting) entropy of $V$ are given by
$$
\operatorname{dens}(V)=\prod_{b\in\mathscr B}\left(1-\frac{1}{b^n}\right)
$$
and
$
h_\T(V)=\log (2)\operatorname{dens}(V),
$
respectively. Again, $\mathbb X_V$ contains the \emph{admissible}\/
subsets $A$ of $\varLambda$, i.e.\
$$
|A_b|<b^{n}
$$
for any $b\in\mathscr B$, where $A_b$ denotes the reduction of $A$ modulo
$b\varLambda$. Moreover, the diffraction measure $\widehat \gamma$
of $V$ exists. It is a pure point measure that is supported on the set of points $\ell\in\mathbb Q\varLambda^*$
with the property that the denominator $\operatorname{den}(\ell)$ divides
a finite product of distinct $b$'s, the intensity at such a point
being given by
$$
I(\ell)=\Bigg(\operatorname{dens}(V)\prod_{\substack{ b\in\mathscr B\\\operatorname{gcd}(\operatorname{den}(\ell),b)\neq 1}}\frac{1}{b^{n}-1}\Bigg)^2.
$$
Both the pure point nature and the intensity formula can be shown by
writing $\widehat \gamma$ as a vague limit of diffraction measures of
approximating crystallographic systems. A Weierstra{\ss} M-test, as
in~\cite{BMP}, implies that one also has the stronger norm
convergence, which preserves the pure point nature of approximating
measures in the limit via~\cite[Thm.\ 8.4]{TAO}.

Further, the frequency of a $\rho$-patch $\mathcal P$ of $V$ is
positive and given by the expression
$$
\nu(\mathcal P)=\sum_{\mathcal F\subset (\varLambda\cap B_{\rho}(0))\setminus \mathcal P}(-1)^{|\mathcal F|}
\prod_{b\in\mathscr B}\left(1-\frac{|(\mathcal P\cup\mathcal
    F)_b|}{b^{n}}\right).
$$ 

The associated measure-theoretic dynamical
system $(\mathbb X_{V},\varLambda,\nu)$ can be seen, as above,
to be ergodic and metrically isomorphic to $(G,\varLambda,\mu)$, where $G$ is the
compact Abelian group
$$
G=\prod_{b\in\mathscr B}\varLambda_b=\prod_{b\in\mathscr B}\varLambda/b\varLambda
$$
on which the lattice $\varLambda$ acts via addition of
$\iota(x)=([x]_b)$
and $\mu$ is the normalised Haar measure on $G$. Again, $V$ turns out
to be $\nu$-generic. Due to the similarity of the structures, we leave
all details to the reader.

\section{$k$-free integers in number fields}\label{number}

Another possible extension is given by the number-theoretic setting that was studied
by Cellarosi and Vinogradov in~\cite{CV}. Let $K$ be an algebraic
number field, i.e.\ a finite extension of $\mathbb Q$, say of degree
$[K:\mathbb Q]=d$; see~\cite{BS,Neukirch} for background material on
algebraic number theory. Further, consider the Dedekind domain $\mathcal O_K$ of
integers in $K$ (in particular, any ideal $0\neq\mathfrak a\neq
\mathcal O_K$
of $\mathcal O_K$ 
factors uniquely as a product of prime ideals) and define, for $k\ge
2$, the set $V$ of \emph{$k$-free integers}\/ of $K$ as the nonzero elements $a\in
\mathcal O_K$ with the property that, if $a$ is not a
unit, then the prime ideal factorisation of
the corresponding principal ideal
$0\neq(a)\neq\mathcal O_K$ contains no $k$th powers of prime ideals, i.e.\ $(a)\not\subseteq
\mathfrak p^k$ for any prime ideal $\mathfrak p\subset \mathcal
O_K$. In other words, one has
$$
V\,=\,V(\mathcal O_K,k)\,=\,\mathcal O_K\setminus\bigcup_{\substack{\mathfrak p\subset \mathcal
    O_K\\\text{\scriptsize $\mathfrak p$ prime ideal}}} \mathfrak p^k.
$$  
It is well known that $\mathcal
O_K$ is a free $\mathbb Z$-module of rank $d$ and is thus isomorphic to the lattice
$\mathbb Z^d$ as a group. In particular, there is a natural isomorphism from $\mathcal
O_K$ to a lattice in $\mathbb R^d$, namely the \emph{Minkowski
  embedding}; see~\cite{BS,Neukirch} and~\cite[Ch.\ 3.4]{TAO}. In order to illustrate this, we prefer to discuss the specific real quadratic
number field  
$K=\mathbb Q(\sqrt 2)$ with $\mathcal
O_K=\mathbb Z[\sqrt 2]$. It is well known that $K$ is a Galois
extension and thus has precisely two field automorphisms, namely the
identity and the non-trivial automorphism determined by $\sqrt
2\mapsto -\sqrt 2$. We denote the latter automorphism by $x\mapsto
x'$. One can then easily check that $\mathcal O_K$ corresponds
under the \emph{Minkowski embedding}\/ $j\!:\, K\rightarrow
\mathbb R^2$, given by
$$
x\mapsto (x,x'),
$$ 
to a non-unimodular lattice
$\mathcal L$ in $\mathbb R^2$ with area $|\det \mathcal L|=2\sqrt
2=\sqrt{ |d_K|}$, where $d_K=8$ is the discriminant of $K$. In fact, since $\mathcal O_K=\mathbb Z\oplus \mathbb Z\sqrt 2$, one has $\mathcal L=\mathbb
Z(1,1)\oplus\mathbb Z(\sqrt 2,-\sqrt 2)$; see Figure~\ref{fig: minkowski}. 

\begin{center}
\begin{figure}
\includegraphics[width=0.85\textwidth]{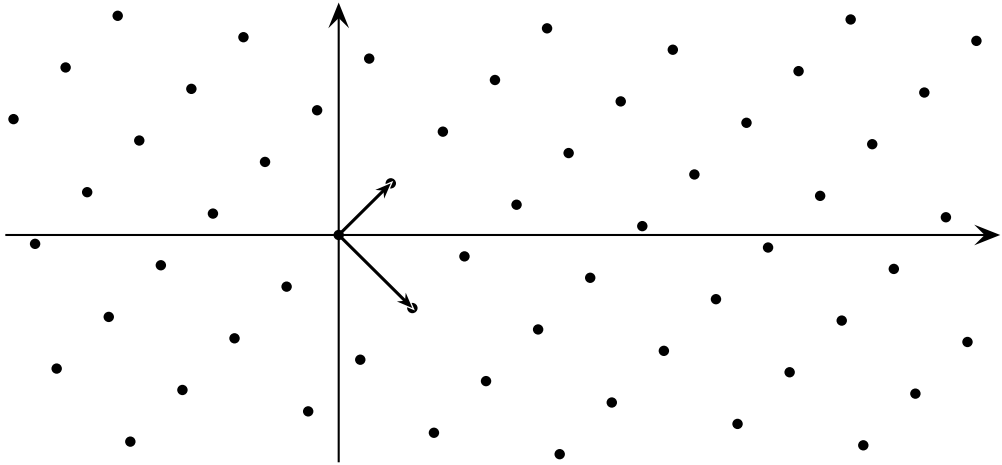}
\caption{The Minkowski embedding of $\mathbb Z[\sqrt 2]$ in $\mathbb R^2$.}
\label{fig: minkowski}
\end{figure}
\end{center}

Moreover, the image
$j(\mathfrak a)$ of any ideal
$0\neq\mathfrak a\subset \mathcal O_K$ is a lattice in $\mathbb
R^2$ with area
\begin{equation}\label{normindex}
|\det j(\mathfrak a)|=2\sqrt 2\, (\mathcal O_K:\mathfrak a),
\end{equation}
where $(\mathcal O_K:\mathfrak a)$ denotes the (finite) subgroup index of
$\mathfrak a$ in $\mathcal O_K$, i.e.\ the absolute norm $N(\mathfrak a)$
of $\mathfrak a$. Note that the absolute norm is a totally multiplicative function on the set of
non-zero ideals of $\mathcal O_K$. We are thus led to consider
the familiar object
$$
j(V)\,=\,\mathcal L\setminus\bigcup_{\mathfrak p} j(\mathfrak p^k),
$$ 
where here and below $\mathfrak p$ runs through the prime ideals of
$\mathcal O_K$; see Fig.~\ref{fig: squarefreenumber} for an illustration. Again, the
proofs in~\cite{BMP,PH} can be adjusted to obtain similar results as above.

\begin{center}
\begin{figure}
\includegraphics[width=\textwidth]{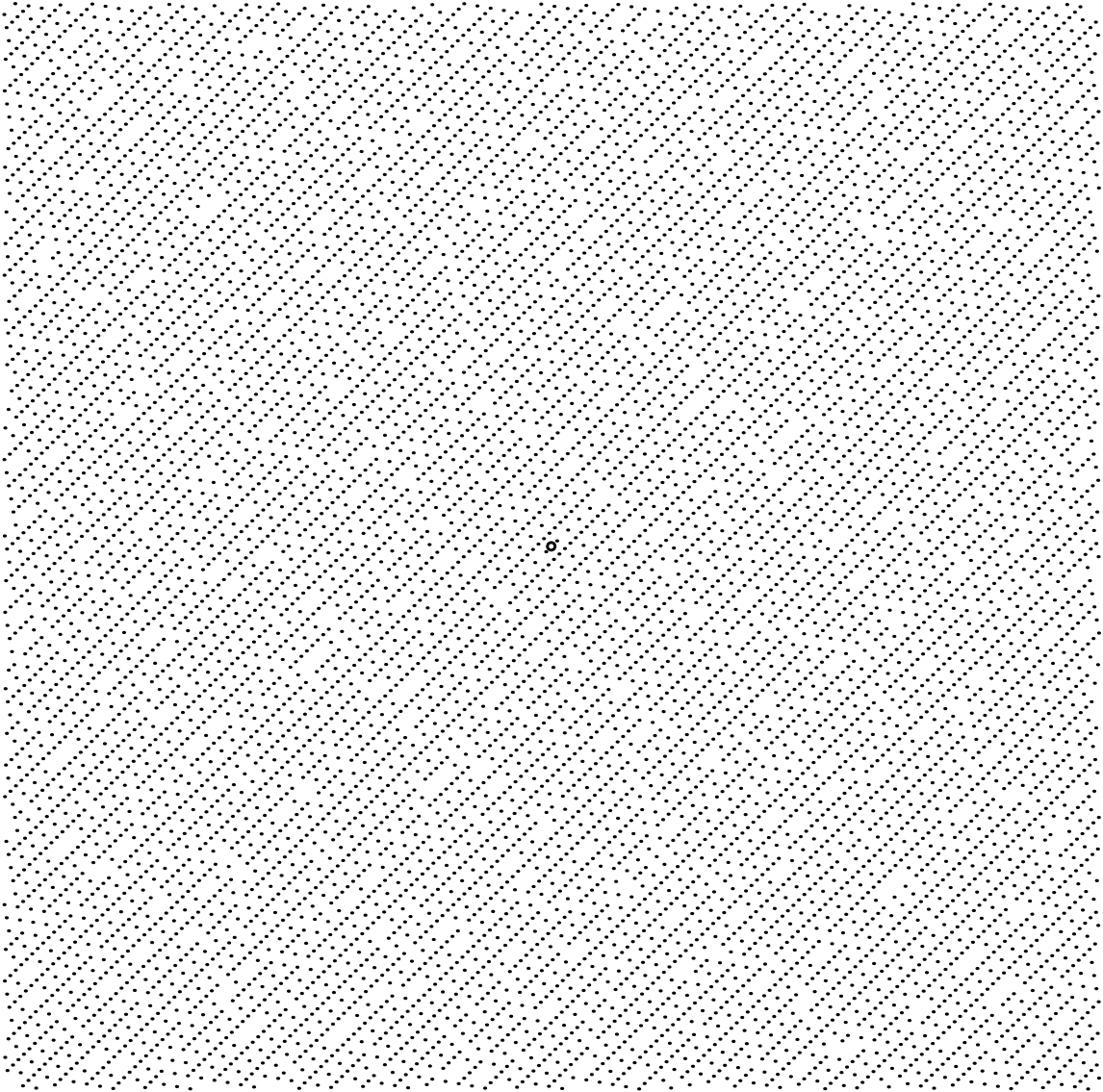}
\caption{A central patch of the Minkowski embedding of the square-free
  integers $V$ of $\mathbb Q(\sqrt 2)$.}
\label{fig: squarefreenumber}
\end{figure}
\end{center}

Denote by 
$$
\zeta_K(s)=\zeta_{\mathbb
  Q(\sqrt 2)}(s)=\sum_{\substack{\mathfrak a\subset \mathcal
    O_K\\\text{\scriptsize $\mathfrak a\neq 0$
      ideal}}}\frac{1}{N(\mathfrak a)^s}=\prod_{\mathfrak p}\left(1-\frac{1}{N(\mathfrak p)^s}\right)^{-1}
$$
the \emph{Dedekind $\zeta$-function}\/ of $K$, which converges for all $s$ with
  $\operatorname{Re}(s)>1$. Employing Eq.~\eqref{normindex}, a similar reasoning as in the
  previous sections now shows that the density and topological (patch
  counting) entropy of $j(V)$ are given by
$$
\operatorname{dens}\big(j(V)\big)=\frac{1}{2\sqrt
  2}\,\frac{1}{\zeta_{K}(k)}=\frac{1}{2\sqrt 2}\prod_{\mathfrak p}\left(1-\frac{1}{N(\mathfrak p)^k}\right)
$$ 
and $\log (2) \operatorname{dens}(j(V))$, respectively. Note that the Chinese Remainder Theorem in its general form says that, given
pairwise coprime ideals $\mathfrak a_1,\dots,\mathfrak a_r$ in a ring $\mathcal
O$ ($\mathfrak a_s+\mathfrak a_t=\mathcal O$ for $s\neq t$), one has
$\prod_{i=1}^r\mathfrak a_i=\mathfrak a_1\cap\dots\cap \mathfrak a_r$
and 
$$
\mathcal O\big/\prod_{i=1}^r\mathfrak a_i\,\simeq\,
\prod_{i=1}^r \mathcal O/\mathfrak a_i.
$$
Recall that the \emph{dual}\/ or \emph{reciprocal 
module}\/ $\mathcal O_K^*$ of $\mathcal O_K$ is the fractional ideal
\[ 
\mathcal O_K^*:=\{y \in K\,\mid\, \operatorname{Tr}_{K/\mathbb Q}(yx)\in\mathbb Z
\mbox{ for all } x\in\mathcal O_K\}
\]
of $K$ containing $\mathcal O_K$, where $\operatorname{Tr}_{K/\mathbb Q}(yx)=yx+y'x'$ is the \emph{trace}\/ of $yx$. Then $j(\mathcal O_K^*)=\mathcal L^*$ and hence $j(\mathbb
Q\mathcal O_K^*)=\mathbb Q
\mathcal L^*$. Here, one calculates that $\mathcal O_K^*=\mathbb Z
\frac{1}{2}\oplus\mathbb Z\frac{\sqrt 2}{4}$ and thus $\mathbb Q\mathcal
O_K=\mathbb Q \mathcal O_K^*=K$ as well as $\mathbb Q\mathcal
L^*=\mathbb Q\mathcal L$. Further, the
\emph{denominator} of a point $\ell$ in $\mathbb
Q\mathcal L^*$ is defined as the non-zero ideal
$$
\operatorname{den}(\ell):=\{x\in \mathcal O_K \,\mid\,x
j^{-1}(\ell)\in\mathcal O_K^*\}\subset\mathcal O_K.
$$
Then, the diffraction measure $\widehat \gamma$
of $j(V)$ is pure point (for the same reason as above) and is supported on the set of points $\ell\in j(\mathbb
Q\mathcal O_K^*)=\mathbb Q
\mathcal L^*$
with $(k+1)$-free denominator $\operatorname{den}(\ell)$ (i.e.,
either $\operatorname{den}(\ell)=\mathcal O_K$ or the unique
prime ideal factorization of $\operatorname{den}(\ell)$ contains no
$(k+1)$th powers), the intensity
at such a point
being given by
$$
I(\ell)=\Bigg(\frac{1}{2\sqrt 2}\,\frac{1}{\zeta_{K}(k)}\prod_{\substack{\mathfrak p\\ \operatorname{den}(\ell)\subset\mathfrak p}}\frac{1}{N(\mathfrak p)^{k}-1}\Bigg)^2.
$$
See Fig.~\ref{fig: squarefreenumberdiff} for an illustration, where
the restriction of $\widehat \gamma$ to a compact region is shown.

\begin{center}
\begin{figure}
\includegraphics[width=0.85\textwidth]{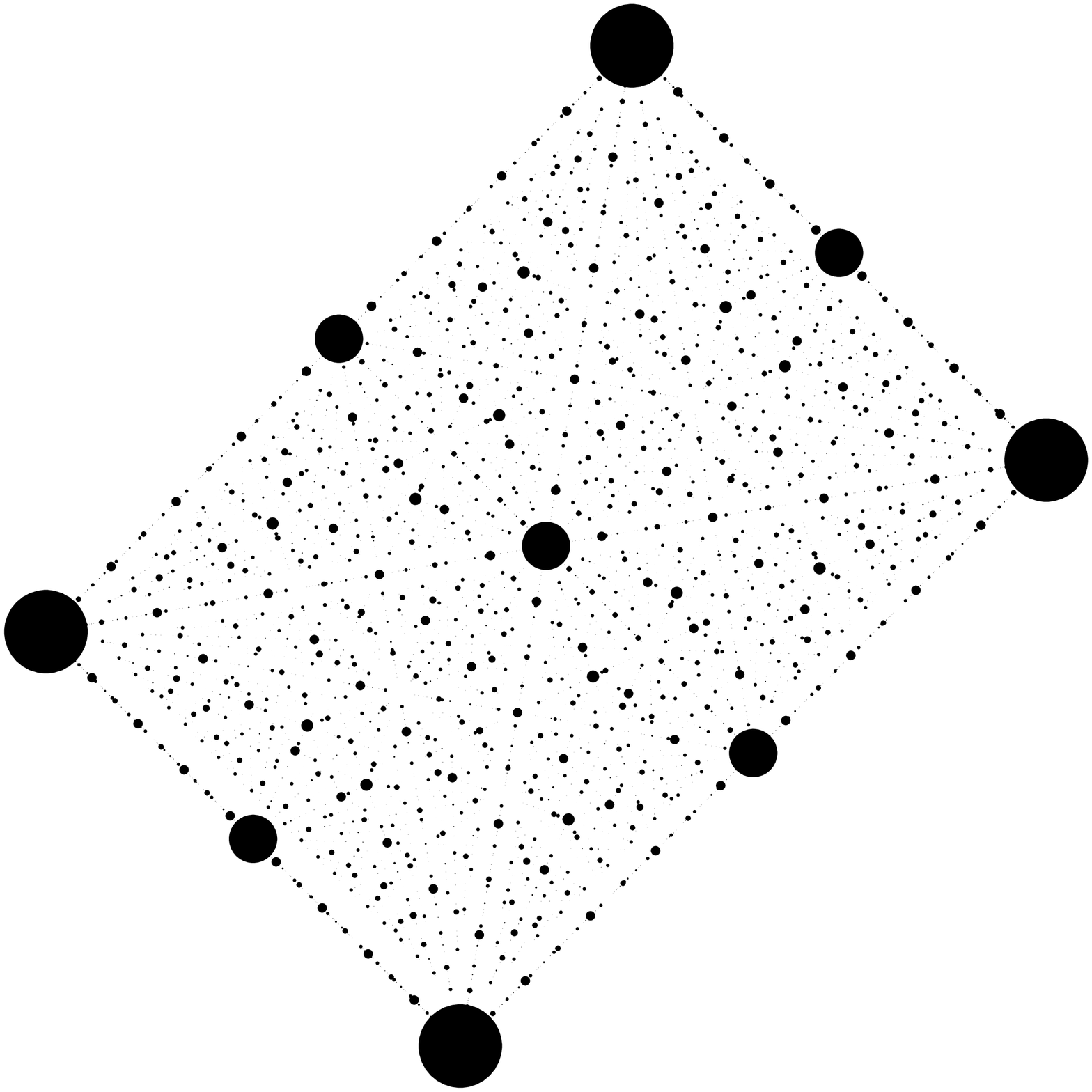}
\caption{Diffraction $\widehat{\gamma}$ of the Minkowski embedding of
  $V=V(\mathbb Z[\sqrt 2], 2)$, with the intensities rescaled by the
  function $x\mapsto  \sqrt[4]{x}/20$ for better
  visibility of small intensities. Its lattice of periods is $\mathcal
  L^*$ as described in the text and the picture shows the intensities
  inside the closure of one fundamental domain.}
\label{fig: squarefreenumberdiff}
\end{figure}
\end{center}

With a view to the general case of an algebraic number field $K$, the
above did not make use of the additional fact that $\mathbb Z[\sqrt
2]$ is in fact a \emph{Euclidean domain}\/ and thus a \emph{principal
  ideal domain}\/ (in particular, a \emph{unique factorisation domain}). Here, one has
$N(\mathfrak a)=|N_{K/\mathbb Q}(a)|$ for $\mathfrak a=(a)$, where
$N_{K/\mathbb Q}(a)=aa'$ is the \emph{norm} of $a$. Note that $\mathbb Z[\sqrt
2]$ is a Euclidean domain with respect to the norm function
$x\mapsto |N_{K/\mathbb Q}(x)|$, i.e. $a+b\sqrt 2\mapsto |a^2-2b^2|$. Furthermore, the
Dedekind $\zeta$-function of $K$ can
be written more explicitly in terms of the usual rational
primes, i.e.\
$$
\zeta_K(s)=\frac{1}{1-2^{-s}}\prod_{p\equiv\pm 1 (8)}\frac{1}{(1-p^{-s})^2}\prod_{p\equiv\pm 3 (8)}\frac{1}{1-p^{-2s}};
$$
see~\cite[Eq.\ (7)]{BM}. For instance, one has
$\zeta_K(2)=\frac{\pi^4}{48\sqrt 2}$; see~\cite[Eq.\ (58)]{BM}. Hence, the intensity $I(\ell)$ can be
computed explicitly in terms of the prime elements of $\mathcal O_K$
dividing any generator of the (principal) ideal
$\operatorname{den}(\ell)$. 

Let us finally turn to the associated topological dynamical system
$(\mathbb X_V,\mathcal O_K)\simeq(\mathbb
X_{j(V)},\mathcal L)$, where $\mathbb X_V$ resp.\ $\mathbb
X_{j(V)}$ are given as the translation orbit closure of $V$ resp.\ $j(V)$ with
respect to the product topology on $\{0,1\}^{\mathcal O_K}$ resp.\
$\{0,1\}^{\mathcal L}$. Again, $\mathbb X_V$ resp.\ $\mathbb
X_{j(V)}$ can be characterised as the \emph{admissible}\/ subsets $A$ of
$\mathcal O_K$ resp.\ $\mathcal L$, i.e.\
$$
|A_{\mathfrak p}|<N(\mathfrak p)^k
$$
for any prime ideal $\mathfrak p$ of $\mathcal O_K$, where
$A_{\mathfrak p}$ denotes the reduction of $A$ modulo $\mathfrak p^k$
resp.\ $j(\mathfrak p^k)$. Further, the frequency of a $\rho$-patch $\mathcal P$ of $j(V)$ is
positive and given by the expression
$$
\nu(\mathcal P)=\sum_{\mathcal F\subset (\mathcal L\cap B_{\rho}(0))\setminus \mathcal P}(-1)^{|\mathcal F|}
\prod_{\mathfrak p}\left(1-\frac{|(\mathcal P\cup\mathcal
    F)_{\mathfrak p}|}{N(\mathfrak p)^k}\right).
$$ 

The associated measure-theoretic dynamical
system $(\mathbb X_{j(V)},\mathcal L,\nu)\simeq (\mathbb X_V,\mathcal O_K,\nu)$ can be seen, as above,
to be ergodic and metrically isomorphic to $(G,\mathcal L,\mu)$, where $G$ is the
compact Abelian group
$$
G=\prod_{\mathfrak p}\mathcal L_{\mathfrak p}=\prod_{\mathfrak
  p}\mathcal L/j(\mathfrak p^k)\simeq \prod_{\mathfrak p}\mathcal
O_K/\mathfrak p^k=\prod_{\mathfrak p}(\mathcal O_K)_{\mathfrak p}
$$
on which the lattice $\mathcal L$ (resp.\ the group $\mathcal O_K$) acts via
addition of $\iota(j(x))=([j(x)]_{\mathfrak p})$
(resp.\ $\iota(x)=([x]_{\mathfrak p})$), where $x\in\mathcal O_K$,   
and $\mu$ is the normalised Haar measure on $G$. As above, $V$ turns out
to be $\nu$-generic.

Since none of the above uses special properties of the quadratic field
$\mathbb Q(\sqrt 2)$, similar results hold for the general case of an arbitrary
algebraic number field $K$. Moreover, even the extension to $\mathscr
B$-free integers in $K$, i.e.\
$$
V\,=\,V_{\mathcal O_K}\,=\,\mathcal O_K\setminus\bigcup_{\mathfrak
  b\in \mathscr B} \mathfrak b,
$$
where $\mathscr B$ is a set of pairwise coprime ideals $\mathfrak
b\subsetneq \mathcal O_K$ satisfying
$$
\sum_{\mathfrak
  b\in \mathscr B}\frac{1}{N(\mathfrak b)}<\infty,
$$
should be possible. 

\appendix
\renewcommand{\theequation}{A\arabic{equation}}
\setcounter{equation}{0}

\section{Ergodicity of the patch frequency measure}\label{appa}

Below, we shall only treat the paradigmatic case $V=V_{\mathbb Z^2}$ from Section~\ref{visible}. For a $\rho$-patch $\mathcal P\in\mathcal A(\rho)$ of $V$ (i.e.\ $P\subset
B_{\rho}(0)\cap \mathbb Z^2$ with $P\in\mathbb A$), denote by
$B_\mathcal P$ the set of elements of $\mathbb X_V=\mathbb A$ whose
$\rho$-patch at $0$ contains $\mathcal P$. One readily checks that the
sets of type 
$B_\mathcal P$ form a semi-algebra that also generates the Borel
$\sigma$-algebra on $\mathbb X_V$. In fact, one has $$C_\mathcal
P=B_\mathcal P\setminus {\bigcup_{\substack{\mathcal Q\in \mathcal
      A(\rho)\\\mathcal P\subsetneq\mathcal Q}}} B_\mathcal Q,$$ and 
\begin{equation}\label{bc}
B_\mathcal P=\dot{\bigcup_{\substack{\mathcal Q\in \mathcal
      A(\rho)\\\mathcal P\subset\mathcal Q}}} C_\mathcal Q.
\end{equation}

\begin{corollary}\label{nub}
For any\/ $\rho$-patch\/ $\mathcal P$ of $V$, one has
$$
\nu(B_\mathcal P)=\prod_p\left(1-\frac{|\mathcal
    P_p|}{p^{2}}\right).
$$
\end{corollary}
\begin{proof}
Let $\mathcal P$ be a $\rho$-patch of $V$. In this proof, we indicate summation variables by a dot under
   the symbol. By Theorem~\ref{freq} and the definition of $\nu$, ~\eqref{bc} implies that
\begin{eqnarray*}
\nu(B_\mathcal P) 
&=& \sum_{\substack{\mathcal Q\in \mathcal A(\rho)\\\mathcal P\subset\mathcal Q}} \nu(C_\mathcal Q)\\
&=&\sum_{\substack{\mathcal Q\in \mathcal A(\rho)\\\mathcal P\subset\mathcal Q}}\hspace{1mm}\sum_{\mathcal Q\subset\mathcal F\subset \mathbb Z^2\cap B_{\rho}(0)}(-1)^{|\mathcal F\setminus \mathcal Q|}\prod_p\left(1-\frac{|\mathcal F_p|}{p^{2}}\right)\\
&=& \sum_{\mathcal P\subset \udo{\mathcal Q}\subset\udo{\mathcal F}\subset \mathbb Z^2\cap B_{\rho}(0)}(-1)^{|\mathcal F\setminus \mathcal Q|}\prod_p\left(1-\frac{|\mathcal F_p|}{p^{2}}\right)\\
&=& \prod_p\left(1-\frac{|\mathcal P_p|}{p^{2}}\right)+\sum_{\substack{\mathcal P\subset \udo{\mathcal Q}\subset\udo{\mathcal F}\subset \mathbb Z^2\cap B_{\rho}(0)\\\mathcal F\neq\mathcal P}}(-1)^{|\mathcal F\setminus \mathcal Q|}\prod_p\left(1-\frac{|\mathcal F_p|}{p^{2}}\right)\\
&=& \prod_p\left(1-\frac{|\mathcal P_p|}{p^{2}}\right),
\end{eqnarray*}
since, for fixed $\mathcal F \subset \mathbb Z^2\cap B_{\rho}(0)$ with
$\mathcal P\subsetneq \mathcal F$, one indeed has
\begin{eqnarray*}
\sum_{\mathcal P\subset \udo{\mathcal Q}\subset\mathcal F}(-1)^{|\mathcal F\setminus \mathcal Q|}\prod_p\left(1-\frac{|\mathcal F_p|}{p^{2}}\right)
&=& \prod_p\left(1-\frac{|\mathcal F_p|}{p^{2}}\right)\sum_{\mathcal P\subset \udo{\mathcal Q}\subset\mathcal F}(-1)^{|\mathcal F\setminus \mathcal Q|}\\
&=&\prod_p\left(1-\frac{|\mathcal F_p|}{p^{2}}\right)\sum_{\udo{\mathcal R}\subset \mathcal F\setminus\mathcal P}(-1)^{|\mathcal R|}\\
&=&\prod_p\left(1-\frac{|\mathcal F_p|}{p^{2}}\right)\sum_{i=0}^{|\mathcal F\setminus\mathcal P|} {|\mathcal F\setminus\mathcal P|\choose i}(-1)^i\\
&=& 0,
\end{eqnarray*}
where the last equality follows from the binomial theorem since $\mathcal F\setminus\mathcal P\neq \varnothing$.
\end{proof}

For a natural number $m$, finite subsets $\mathcal P$ and  $\mathcal Q$ of
$\mathbb Z^2$ and $S\subset \mathcal P_m$, we set
$$
\mathcal Q_{S,\mathcal P}^m=\Big(\bigcap_{s\in S}\mathcal
Q_m-s\Big)\setminus\Big(\bigcup_{s\in\mathcal P_m\setminus S}\mathcal Q_m-s\Big),
$$
i.e.\ the set of elements of $(\mathbb Z^2)_m$ that lie
in $\mathcal Q_m-s$ precisely for those $s\in \mathcal P_m$ with $s\in
S\subset\mathcal P_m$, in particular $\mathcal
Q_{\varnothing,\mathcal P}^m=(\mathbb Z^2)_m \setminus (\mathcal
Q_m-\mathcal P_m)$. With $q_{S,\mathcal P}^m=|\mathcal Q_{S,\mathcal P}^m|$, one then has
$q_{\varnothing,\mathcal P}^m=m^{2}-|\mathcal Q_m-\mathcal P_m|$ and, since
the difference set $\mathcal Q_m-\mathcal P_m$
 is the disjoint union of the various $\mathcal Q_{S,\mathcal P}^m$,
 where one has 
 $\varnothing\neq S\subset \mathcal P_m$,
\begin{equation}\label{eq1}
\sum_{S\subset \mathcal P_m}q_{S,\mathcal P}^m=m^{2}.
\end{equation}

Note that the following two lemmas also hold for any
finite
subsets $\mathcal P$ and $\mathcal Q$ of an arbitrary finite group $G$ instead of 
$G=(\mathbb Z^2)_m$.  

\begin{lemma}\label{lambda}
For any finite subsets\/ $\mathcal P$ and\/ $\mathcal Q$ of\/
$\mathbb Z^2$ and any natural number\/ $m$, one has
$$
\sum_{S\subset \mathcal P_m}|S|q_{S,\mathcal P}^m=|\mathcal
    P_m||\mathcal
    Q_m|.
$$
\end{lemma}
\begin{proof}
We use induction on $|\mathcal
    P_m|$. For the induction basis $|\mathcal
    P_m|=0$, the assertion is trivially true. For the induction step, consider $\mathcal P$ with $|\mathcal P_m|>0$
and fix an element, say $*$, of $\mathcal P_m$. It follows that
\begin{eqnarray*}
\sum_{S\subset \mathcal P_m}|S|q_{S,\mathcal P}^m
&=&\sum_{S\subset \mathcal P_m\setminus\{*\}}|S|q_{S,\mathcal
  P}^m+\sum_{S\subset \mathcal
  P_m\setminus\{*\}}(|S|+1)q_{S\cup\{*\},\mathcal P}^m\\
&=&\sum_{S\subset \mathcal
  P_m\setminus\{*\}}|S|(q_{S,\mathcal
  P}^m+q_{S\cup\{*\},\mathcal P}^m)+\sum_{S\subset
  \mathcal P_m\setminus\{*\}}q_{S\cup\{*\},\mathcal P}^m\\
&=&\sum_{S\subset \mathcal
  P_m\setminus\{*\}}|S|(q_{S,\mathcal
  P}^m+q_{S\cup\{*\},\mathcal P}^m)+|\mathcal Q_m|,
\end{eqnarray*}
since 
\begin{eqnarray*}
\sum_{S\subset
  \mathcal P_m\setminus\{*\}}q_{S\cup\{*\},\mathcal P}^m
&=&\sum_{S\subset
  \mathcal P_m\setminus\{*\}}\left|\left((\mathcal Q_m-\{*\})\cap \Big(\bigcap_{s\in S}\mathcal
Q_m-s\Big)\right)\setminus\Big(\bigcup_{s\in\mathcal P_m\setminus
  (S\cup\{*\})}\mathcal Q_m-s\Big)\right |\\
&=&\sum_{S\subset
  \mathcal P_m\setminus\{*\}}\left|(\mathcal Q_m-\{*\})\cap \left(\Big(\bigcap_{s\in S}\mathcal
Q_m-s\Big)\setminus\Big(\bigcup_{s\in (\mathcal P_m\setminus\{*\})\setminus
  S}\mathcal Q_m-s\Big)\right)\right |\\
&=&|\mathcal Q_m-\{*\}|\\
&=&|\mathcal Q_m|.
\end{eqnarray*}
Furthermore, for $S\subset \mathcal P_m\setminus\{*\}$, one has
\begin{eqnarray*}
&&\hspace{-2em}
\mathcal Q_{S\cup\{*\},\mathcal P}^m\cup \mathcal Q_{S,\mathcal P}^m\\
&=&\left[(\mathcal Q_m-\{*\})\cap \left(\Big(\bigcap_{s\in S}\mathcal
Q_m-s\Big)\setminus\Big(\bigcup_{s\in (\mathcal P_m\setminus\{*\})\setminus
  S}\mathcal Q_m-s\Big)\right)\right]\dot\bigcup \\
&\hphantom{=}&\left[\left(\Big(\bigcap_{s\in S}\mathcal
Q_m-s\Big)\setminus (\mathcal Q_m-\{*\})\right)\cap \left(\Big(\bigcap_{s\in S}\mathcal
Q_m-s\Big)\setminus\Big(\bigcup_{s\in (\mathcal P_m\setminus\{*\})\setminus
  S}\mathcal Q_m-s\Big)\right)\right]\\
&=& \Big(\bigcap_{s\in S}\mathcal
Q_m-s\Big)\setminus\Big(\bigcup_{s\in (\mathcal P_m\setminus\{*\})\setminus
  S}\mathcal Q_m-s\Big)\\
&=&\mathcal Q_{S,\mathcal P\setminus \{*\}}^m
\end{eqnarray*}
wherefore, by the induction hypothesis,
$$
\sum_{S\subset \mathcal
  P_m\setminus\{*\}}|S|(q_{S,\mathcal P}^m+q_{S\cup\{*\},\mathcal P}^m)=\sum_{S\subset \mathcal
  P_m\setminus\{*\}}|S|q_{S,\mathcal P\setminus \{*\}}^m=(|\mathcal P_m|-1)|\mathcal
Q_m|.
$$
This completes the proof.
\end{proof}

\begin{lemma}\label{lambda2}
For any finite subsets\/ $\mathcal P$ and\/ $\mathcal Q$ of
$\mathbb Z^2$ and square-free
$d$, one has
$$
\sum_{\substack{(\nu_p)_{p\mid
      d}\\0\le\nu_p\le|\mathcal P_p|}}\prod_{p\mid
  d}\Big((\nu_p+|\mathcal Q_p|)\sum_{\substack{S\subset \mathcal
    P_p\\|S|=|\mathcal P_p|-\nu_p}}q_{S,\mathcal P}^p\Big)\,=\,\prod_{p\mid d}\left(p^{2}|\mathcal P_p|+ p^{2}|\mathcal
    Q_p|-|\mathcal P_p| |\mathcal Q_p|\right).
$$
\end{lemma}
\begin{proof}
This is proved by induction on the number $\omega(d)$ of prime
factors of $d$. For the induction basis $\omega(d)=1$, say $d=p$, note that by~\eqref{eq1}
and Lemma~\ref{lambda} one indeed has
\begin{eqnarray*}
\sum_{0\le\nu_p\le|\mathcal P_p|}\Big((\nu_p+|\mathcal
Q_p|)\sum_{\substack{S\subset \mathcal P_p\\|S|=|\mathcal
    P_p|-\nu_p}}q_{S,\mathcal P}^p\Big)
&=& \Big(|\mathcal
Q_p| \sum_{S\subset \mathcal P_p}q_{S,\mathcal P}^p \Big)+\sum_{S\subset \mathcal P_p}(|\mathcal
    P_p|-|S|)q_{S,\mathcal P}^p\\
&=& |\mathcal
Q_p|p^{2}+|\mathcal
P_p|p^{2}-|\mathcal
    P_p||\mathcal
    Q_p|.
\end{eqnarray*}
For the induction step, consider $d$ with $\omega(d)=r+1$, where $r\ge
1$, say $d=p_1\ldots p_rp_{r+1}$. Then
\begin{eqnarray*}
&&\hspace{-2em}
\sum_{\substack{(\nu_p)_{p\mid
      d}\\0\le\nu_p\le|\mathcal P_p|}}\prod_{p\mid
  d}\Big((\nu_p+|\mathcal Q_p|)\sum_{\substack{S\subset \mathcal
    P_p\\|S|=|\mathcal P_p|-\nu_p}}q_{S,\mathcal P}^p\Big)\\
&=&\sum_{i=0}^{|\mathcal P_{p_{r+1}}|}\sum_{\substack{(\nu_p)_{p\mid
      d}\\0\le\nu_p\le|\mathcal P_p|\\ \nu_{p_{r+1}}=i}}(i+|\mathcal Q_{p_{r+1}}|)\sum_{\substack{S\subset \mathcal
    P_{p_{r+1}}\\|S|=|\mathcal P_{p_{r+1}}|-i}}q_{S,\mathcal P}^{p_{r+1}}\prod_{p\mid
  \frac{d}{p_{r+1}}}\Big((\nu_p+|\mathcal Q_p|)\sum_{\substack{S\subset \mathcal
    P_p\\|S|=|\mathcal P_p|-\nu_p}}q_{S,\mathcal P}^p\Big)\\
&=&\sum_{i=0}^{|\mathcal P_{p_{r+1}}|}(i+|\mathcal Q_{p_{r+1}}|)\sum_{\substack{S\subset \mathcal
    P_{p_{r+1}}\\|S|=|\mathcal P_{p_{r+1}}|-i}}q_{S,\mathcal P}^{p_{r+1}}\sum_{\substack{(\nu_p)_{p\mid
      \frac{d}{p_{r+1}}}\\0\le\nu_p\le|\mathcal P_p|}}\prod_{p\mid
  \frac{d}{p_{r+1}}}\Big((\nu_p+|\mathcal Q_p|)\sum_{\substack{S\subset \mathcal
    P_p\\|S|=|\mathcal P_p|-\nu_p}}q_{S,\mathcal P}^p\Big),
\end{eqnarray*}
which is, by the induction hypothesis,
\begin{eqnarray*}
&&\hspace{-2em}
\sum_{i=0}^{|\mathcal P_{p_{r+1}}|}(i+|\mathcal Q_{p_{r+1}}|)\sum_{\substack{S\subset \mathcal
    P_{p_{r+1}}\\|S|=|\mathcal P_{p_{r+1}}|-i}}q_{S,\mathcal P}^{p_{r+1}}\prod_{p\mid \frac{d}{p_{r+1}}}\left(p^{2}|\mathcal P_p|+ p^{2}|\mathcal
    Q_p|-|\mathcal P_p| |\mathcal Q_p|\right)\\
&=&\prod_{p\mid \frac{d}{p_{r+1}}}\left(p^{2}|\mathcal P_p|+ p^{2}|\mathcal
    Q_p|-|\mathcal P_p| |\mathcal Q_p|\right)\sum_{i=0}^{|\mathcal P_{p_{r+1}}|}\sum_{\substack{S\subset \mathcal
    P_{p_{r+1}}\\|S|=|\mathcal P_{p_{r+1}}|-i}}(i+|\mathcal
Q_{p_{r+1}}|)q_{S,\mathcal P}^{p_{r+1}}.
\end{eqnarray*}
It thus remains to show that
$$
\sum_{i=0}^{|\mathcal P_{p_{r+1}}|}\sum_{\substack{S\subset \mathcal
    P_{p_{r+1}}\\|S|=|\mathcal P_{p_{r+1}}|-i}}(i+|\mathcal
Q_{p_{r+1}}|)q_{S,\mathcal P}^{p_{r+1}}\,=\,p_{r+1}^{2}|\mathcal P_{p_{r+1}}|+ p_{r+1}^{2}|\mathcal
    Q_{p_{r+1}}|-|\mathcal P_{p_{r+1}}| |\mathcal Q_{p_{r+1}}|,
$$
which is clear since, by~\eqref{eq1}
and Lemma~\ref{lambda} again, one has
\begin{eqnarray*}
&&\hspace{-2em}
\sum_{i=0}^{|\mathcal P_{p_{r+1}}|}\sum_{\substack{S\subset \mathcal
    P_{p_{r+1}}\\|S|=|\mathcal P_{p_{r+1}}|-i}}(i+|\mathcal
Q_{p_{r+1}}|)q_{S,\mathcal P}^{p_{r+1}}\\
&=&\sum_{i=0}^{|\mathcal P_{p_{r+1}}|}\sum_{\substack{S\subset \mathcal
    P_{p_{r+1}}\\|S|=|\mathcal P_{p_{r+1}}|-i}}iq_{S,\mathcal P}^{p_{r+1}}+\sum_{i=0}^{|\mathcal P_{p_{r+1}}|}\sum_{\substack{S\subset \mathcal
    P_{p_{r+1}}\\|S|=|\mathcal P_{p_{r+1}}|-i}}|\mathcal
Q_{p_{r+1}}|q_{S,\mathcal P}^{p_{r+1}}\\
&=&\sum_{S\subset \mathcal
    P_{p_{r+1}}}(|\mathcal P_{p_{r+1}}|-|S|)q_{S,\mathcal P}^{p_{r+1}}+|\mathcal
Q_{p_{r+1}}|\sum_{S\subset \mathcal
    P_{p_{r+1}}}q_{S,\mathcal P}^{p_{r+1}}\\
&=&|\mathcal P_{p_{r+1}}|p_{r+1}^{2}-|\mathcal P_{p_{r+1}}| |\mathcal Q_{p_{r+1}}|+ |\mathcal
    Q_{p_{r+1}}|p_{r+1}^{2}.
\end{eqnarray*}
This completes the proof.
\end{proof}

The proof of the main result below needs the well-known estimate
\begin{equation}\label{LambdaCount}
|B_\rho({ x})\cap\mathbb Z^2|=\pi\rho^2+
O(\rho)+O(1),
\end{equation}
approximating the number of points of
$\mathbb Z^2$ in a large ball $B_\rho({ x})$ (the last error term being
required only when $\rho^2<1$). This is
obtained by dividing $B_\rho({ x})$ into fundamental regions for
$\mathbb Z^2$, each of volume $1$ and containing one point of
$\mathbb Z^2$, with the error terms arising from fundamental regions that
overlap the boundary of $B_\rho({ x})$. A more precise version is
given as~\cite[Prop.~1]{BMP}.

\begin{theorem}
The measure\/ $\nu$ is ergodic.
\end{theorem}
\begin{proof}
We have to show that
$$
\nu(B_\mathcal P)\nu(B_\mathcal Q)=\lim_{R\to \infty}\frac{1}{\pi R^2}\sum_{x\in B_R(0)\cap \mathbb Z^2}\nu(x+B_\mathcal P\cap B_\mathcal Q)
$$
for any finite admissible subsets $\mathcal P$ and  $\mathcal Q$ of
$\mathbb Z^2$; cf.~\cite[Thm.~1.17(i)]{Walters}. By Corollary~\ref{nub}, the latter is equivalent to
$$
\prod_p\left(1-\frac{|\mathcal
    P_p|}{p^{2}}\right)\prod_p\left(1-\frac{|\mathcal
    Q_p|}{p^{2}}\right)=\lim_{R\to \infty}\frac{1}{\pi R^2}\sum_{x\in B_R(0)\cap \mathbb Z^2}\nu(x+B_\mathcal P\cap B_\mathcal Q).
$$
Let $\mu$ denote the M\"obius function for the rest of this proof. The left-hand side is equal to
\begin{eqnarray*}
\prod_p\left(1-\frac{|\mathcal P_p|}{p^{2}}- \frac{|\mathcal
    Q_p|}{p^{2}}+\frac{|\mathcal P_p| |\mathcal Q_p|}{p^{4}} \right)&=&\prod_p\left(1-\left(\frac{|\mathcal P_p|}{p^{2}}+ \frac{|\mathcal
    Q_p|}{p^{2}}-\frac{|\mathcal P_p||\mathcal Q_p|}{p^{4}} \right)\right)\\
&=&\sum_{\text{\scriptsize $d$ square-free}}\mu(d)\prod_{p\mid d}\left(\frac{|\mathcal P_p|}{p^{2}}+ \frac{|\mathcal
    Q_p|}{p^{2}}-\frac{|\mathcal P_p| |\mathcal Q_p|}{p^{4}} \right)\\
&=&\sum_{\text{\scriptsize $d$ square-free}}\frac{\mu(d)}{d^{2}}\prod_{p\mid d}\left(|\mathcal P_p|+ |\mathcal
    Q_p|-\frac{|\mathcal P_p| |\mathcal Q_p|}{p^{2}} \right).
\end{eqnarray*}
For a fixed $R>0$, the right-hand side is
\begin{eqnarray*}
&&\hspace{-2em}\frac{1}{\pi R^2}\sum_{x\in B_R(0)\cap
  \mathbb Z^2}\prod_p\left(1-\frac{|(\mathcal Q\cup x+\mathcal
    P)_p|}{p^{2}}\right)\\
&=& \frac{1}{\pi R^2}\sum_{x\in B_R(0)\cap
  \mathbb Z^2}\hspace{1mm}\sum_{\text{\scriptsize $d$
    square-free}}\mu(d)\prod_{p\mid d}\frac{|(\mathcal Q\cup x+\mathcal
    P)_p|}{p^{2}}\\
&=& \frac{1}{\pi R^2}\sum_{\text{\scriptsize $d$
    square-free}}\mu(d)\sum_{x\in B_R(0)\cap
  \mathbb Z^2}\prod_{p\mid d}\frac{|(\mathcal Q\cup x+\mathcal
    P)_p|}{p^{2}}\\
&=& \frac{1}{\pi R^2}\sum_{\text{\scriptsize $d$
    square-free}}\mu(d)\sum_{\substack{(\nu_p)_{p\mid d}\\|\mathcal
    Q_p|\le\nu_p\le|\mathcal Q_p|+|\mathcal P_p|}}\,\sum_{\substack{x\in B_R(0)\cap
  \mathbb Z^2\\\forall p\mid d: |(\mathcal Q\cup x+\mathcal
    P)_p|=\nu_p}}\prod_{p\mid d}\frac{\nu_p}{p^{2}}\\
&=& \frac{1}{\pi R^2}\sum_{\text{\scriptsize $d$
    square-free}}\mu(d)\sum_{\substack{(\nu_p)_{p\mid d}\\|\mathcal
    Q_p|\le\nu_p\le|\mathcal Q_p|+|\mathcal P_p|}}\prod_{p\mid d}\frac{\nu_p}{p^{2}}\sum_{\substack{x\in B_R(0)\cap
  \mathbb Z^2\\\forall p\mid d: |(\mathcal Q\cup x+\mathcal
    P)_p|=\nu_p}}1\\
&=& \frac{1}{\pi R^2}\sum_{\text{\scriptsize $d$
    square-free}}\frac{\mu(d)}{d^{2}}\sum_{\substack{(\nu_p)_{p\mid d}\\|\mathcal
    Q_p|\le\nu_p\le|\mathcal Q_p|+|\mathcal P_p|}}\prod_{p\mid d}\nu_p\sum_{\substack{x\in B_R(0)\cap
  \mathbb Z^2\\\forall p\mid d: |(\mathcal Q\cup x+\mathcal
    P)_p|=\nu_p}}1.\end{eqnarray*}
By~\cite[Prop.~1]{BMP} and Proposition~\ref{crt}
with $\mathbb Z^2$ replaced by the lattices $p\mathbb Z^2$, where
$p\mid d$, one obtains the estimate
$$
\left(\frac{\pi R^2}{d^{2}}+
O\left(\left(\frac{R^2}{d^{2}}\right)^{1-1/2}\right)+O(1)\right)\prod_{p\mid
d}\big|\{x_p\in (\mathbb Z^2)_p\,:\,|(\mathcal Q\cup x+\mathcal
    P)_p|=\nu_p\}\big|
$$
for the inner sum. Substituting this in the above expression and
letting $R$ tend to infinity, one obtains
\begin{eqnarray*}
&&\hspace{-2em}
\sum_{\text{\scriptsize $d$
    square-free}}\frac{\mu(d)}{d^{4}}\sum_{\substack{(\nu_p)_{p\mid d}\\|\mathcal
    Q_p|\le\nu_p\le\mathcal |Q_p|+|\mathcal P_p|}}\prod_{p\mid d}\left(\nu_p\,\big|\{x_p\in (\mathbb Z^2)_p\,:\,|(\mathcal Q\cup x+\mathcal
    P)_p|=\nu_p\}\big|\right)\\
&=&\sum_{\text{\scriptsize $d$
    square-free}}\frac{\mu(d)}{d^{4}}\sum_{\substack{(\nu_p)_{p\mid
      d}\\0\le\nu_p\le|\mathcal P_p|}}\prod_{p\mid
  d}\left((\nu_p+|\mathcal Q_p|)\,\big|\{x_p\in (\mathbb Z^2)_p\,:\,|(\mathcal Q\cup x+\mathcal
    P)_p|=\nu_p+|\mathcal Q_p|\}\big|\right)
\end{eqnarray*}
and the inner product can be rewritten as
\begin{eqnarray*}
&&\hspace{-2em}
\prod_{p\mid
  d}\left((\nu_p+|\mathcal Q_p|)\,\big|\{x_p\in (\mathbb Z^2)_p\,:\,x_p\in \mathcal Q_{S,\mathcal P}^p\mbox{ for
    $S\subset\mathcal P_p$ with $|S|=|\mathcal P_p|-\nu_p$}\}\big|\right)
\\
&=&
\prod_{p\mid
  d}\Big((\nu_p+|\mathcal Q_p|)\sum_{\substack{S\subset \mathcal
    P_p\\|S|=|\mathcal P_p|-\nu_p}}q_{S,\mathcal P}^p\Big).
\end{eqnarray*}

Thus, in order to prove the claim, it suffices to show, for square-free
$d$, the identity
$$
\sum_{\substack{(\nu_p)_{p\mid
      d}\\0\le\nu_p\le|\mathcal P_p|}}\prod_{p\mid
  d}\Big((\nu_p+|\mathcal Q_p|)\sum_{\substack{S\subset \mathcal
    P_p\\|S|=|\mathcal P_p|-\nu_p}}q_{S,\mathcal P}^p\Big)\,=\,\prod_{p\mid d}\left(p^{2}|\mathcal P_p|+ p^{2}|\mathcal
    Q_p|-|\mathcal P_p| |\mathcal Q_p|\right),
$$
which is just the content of Lemma~\ref{lambda2}. 
\end{proof}

\renewcommand{\theequation}{B\arabic{equation}}
\setcounter{equation}{0}

\section{Isomorphism between $(\mathbb X_{V},\mathbb Z^2,\nu)$ and $(\prod_p (\mathbb Z^2)_p,\mathbb Z^2,\mu)$}\label{appb}

Let $\mathbb A_1$ be the Borel subset of $\mathbb X_V=\mathbb A$ consisting
of the elements $X\in\mathbb X_V$ that satisfy
$$
|X_p|=p^2-1
$$
for any prime $p$, i.e.\ $X$ misses exactly one coset of $p\mathbb
Z^2$ in $\mathbb Z^2$. There is a natural Borel map
$$
\theta\!:\,\mathbb A_1\rightarrow \mathbb T,
$$
given by
$X\mapsto ([y_p]_p)$, where $[y_p]_p$ is uniquely determined by $[y_p]_p\not\in
X_p$.  Note that $\theta$ fails to be continuous or injective. Clearly, $\mathbb A_1$ is $\mathbb Z^2$-invariant and
$\theta$ intertwines the $\mathbb Z^2$-actions.  Note also that, for
$X\in\mathbb A_1$, one clearly has 
\begin{equation}\label{phitheta}
X\subset\varphi(\theta(X)).
\end{equation}

\begin{lemma}\label{va1}
One has\/ $V\in\mathbb A_1$.
\end{lemma}
\begin{proof}
Fix a prime number $p$ and choose a set $A$ of $p^2-1$ elements of $\mathbb
Z^2$ such that $|A_p|=p^2-1$. By the Chinese Remainder
Theorem (Prop.\ \ref{crt}), we may assume that $A_{p'}=\{(0,0)_{p'}\}$ for the
finitely many primes $p'<p$. Then, $A\in\mathbb A$ and, by
Theorem~\ref{charachull}, there is a translation $t\in \mathbb Z^2$ such
that $t+A\subset V$. Since 
$$
p^2-1= |(t+A)_p|\le |V_p|\le p^2-1,
$$
the assertion follows.
\end{proof}

\begin{lemma}
One has\/ $\nu(\mathbb A_1)=1$.
\end{lemma}
\begin{proof}
The ergodicity of the full measure
$\nu$ (Thms.~\ref{freq} and~\ref{c2}(b)) implies that $\nu$-almost
every $X\in\mathbb X_V$ has a dense $\mathbb Z^2$-orbit;
compare~\cite[Thm.~1.7]{Walters}. It follows from Lemma~\ref{va1} that
$$
\{X\in\mathbb X_V\,|\,X \mbox{ has a dense $\mathbb Z^2$-orbit}\}\subset
\mathbb A_1.
$$
This inclusion is due to the fact that, if $X$ has a dense
orbit, then in particular $V\in\mathbb A_1$ is
an element of the orbit closure of $X$. Since, for any prime $p$, representatives of the
$p^2-1$ different elements of $V_p$ in $V$ can be chosen within a finite
distance from the origin, it is clear from the definition of the
topology on $\mathbb X_V$ that there is a translation $t\in\mathbb
Z^2$ such that $t+X$ contains these representatives. Thus
$$p^2-1\ge|X_p|=|(t+X)_p|\ge|V_p|=p^2-1$$ and the assertion follows.
\end{proof}

Hence the push-forward measure of $\nu$ to $\mathbb T$ by
the map $\theta$ is a $\mathbb
Z^2$-invariant probability measure and thus is the normalised Haar
measure $\mu$. Set
$$
\mathbb T_1:=\varphi^{-1}(\mathbb A_1),
$$
which is a $\mathbb Z^2$-invariant Borel
set with $\varphi(\mathbb T_1)\subset \mathbb A_1$.  In
particular, this shows that $\theta\circ\varphi|_{\mathbb
  T_1}=\operatorname{id}_{\mathbb T_1}$ and thus the restriction of
$\varphi$ to $\mathbb T_1$ and the restriction of $\theta$ to
$\varphi(\mathbb T_1)$ are
injective. 

\begin{lemma}\label{a1t1}
One has\/ $\theta(\mathbb
A_1)=\mathbb T_1$. 
\end{lemma}
\begin{proof}
It suffices to prove the inclusion $\theta(\mathbb
A_1)\subset\mathbb T_1$, since this immediately yields the assertion due to $\mathbb T_1=\theta(\varphi(\mathbb
T_1))\subset \theta(\mathbb A_1)$. So let us assume the existence of an
$X\in\mathbb A_1$ with $\theta(X)\not\in \mathbb T_1$, i.e.\ 
$\varphi(\theta(X))\not\in\mathbb
A_1$. Then there is a prime number $p$ such that
$$|(\varphi(\theta(X)))_p|<p^2-1.$$ Using~\eqref{phitheta},
this implies that also $|X_p|<p^2-1$, a contradiction. 
\end{proof}

In particular, this shows that $\theta^{-1}(\mathbb T_1)=\mathbb
A_1$ and thus
$$
\mu(\mathbb T_1)=\nu(\mathbb A_1)=1.
$$  
It follows that $\theta$ is
a factor map from $(\mathbb X_{V},\mathbb Z^2,\nu)$ to $(\prod_p
(\mathbb Z^2)_p,\mathbb Z^2,\mu)$. 

In order to see that $\theta$ is in fact an isomorphism, let us first
note that the Borel map $\varphi\!:\,\mathbb T\rightarrow \mathbb
X_V$ is measure-preserving since, for any $\rho$-patch
$\mathcal P$, one has 
\begin{eqnarray*}
\mu(\varphi^{-1}(B_\mathcal P)) &=&\mu\big(\{([y_p]_p)\in\mathbb
T\,|\,[y_p]_p\not\in\mathcal P_p\mbox{ for all } p 
\}\big)\\
&=&\prod_p\left(\frac{p^2-|\mathcal
    P_p|}{p^2}\right)\\
&=&\nu(B_\mathcal P)
\end{eqnarray*}
by Corollary~\ref{nub}. Next, consider the
subset $\mathbb A_1^*$ of elements $X\in\mathbb A_1$ that are maximal
elements of $\mathbb A$ with respect to inclusion, i.e. 
$$
\forall\, Y\in\mathbb A:\,(X\subset Y\Rightarrow
X=Y).
$$ 
Clearly, $V$ and
every translate of $V$ are elements of $\mathbb A_1^*$; see
Lemma~\ref{va1}. Using Lemma~\ref{a1t1}, one can
verify that $\mathbb
A_1^*$ contains precisely the elements $X\in\mathbb A_1$ with
\begin{equation}\label{phitheta2}
X=\varphi(\theta(X)).
\end{equation}
Employing~\eqref{phitheta2}, one further
verifies that $$\mathbb
A_1^*=\mathbb A_1\cap \varphi(\mathbb T).$$ Since $\varphi(\mathbb
T)$ can be seen to be a Borel set, 
$\mathbb A_1^*$ is thus a Borel
set with measure
$$
\nu(\mathbb A_1^*)=\nu(\mathbb A_1\cap \varphi(\mathbb T))=\nu(\mathbb A_1)+\nu(\varphi(\mathbb T))-\nu(\mathbb A_1\cup \varphi(\mathbb T))=1.
$$ 
Setting 
$$
\mathbb T_1^*:=\varphi^{-1}(\mathbb A_1^*),
$$
the restrictions $\varphi\!:\,\mathbb T_1^*\rightarrow \mathbb A_1^*$ and
$\theta\!:\,\mathbb A_1^*\rightarrow\mathbb T_1^*$ are well-defined
($X\in\mathbb A_1^*\Rightarrow \varphi(\theta(X))=X\in\mathbb A_1^*$)
and can now be
shown to be  bijective and
inverses to each other. Hence $\theta$ and 
$\varphi$ are isomorphisms.

\section*{Acknowledgements}
It is our pleasure to thank Francesco Cellarosi, Joanna Ku\l
aga-Przymus, Marius Lema\'nczyk, Daniel Lenz, Peter Sarnak and Rudolf Scharlau for
valuable discussions. Special thanks are due to Tobias Jakobi for
providing Figures~\ref{fig: squarefreenumber} and~\ref{fig: squarefreenumberdiff}. We are grateful to an anonymous
reviewer for a number of careful comments. This work was supported by the German Research Foundation (DFG), within
the CRC~701.

\end{document}